\documentclass[aap]{imsart}

\RequirePackage{amsmath,amsthm,amsfonts,amssymb,mathtools,acronym,dsfont,mathrsfs}
\RequirePackage[numbers]{natbib}
\RequirePackage[colorlinks,citecolor=blue,urlcolor=blue]{hyperref}
\RequirePackage{graphicx}

\startlocaldefs

\numberwithin{equation}{section}  %numberwithin goes before cleverefs when using hyperref
\usepackage[sort&compress,capitalize,nameinlink]{cleveref}
\usepackage{nameref}

\crefname{app}{Appendix}{Appendices}
\crefname{assumption}{Assumption}{Assumptions}
\crefname{theorem}{Theorem}{Theorems}
\crefname{lemma}{Lemma}{Lemmata}
\crefname{proposition}{Proposition}{Proposition}
\crefname{corollary}{Corollary}{Corollaries}
\crefname{remark}{Remark}{Remarks}

\crefrangeformat{equation}{\upshape(#3#1#4)\textendash(#5#2#6)}

\newtheorem{theorem}{Theorem}[section]
\newtheorem{lemma}[theorem]{Lemma}
\newtheorem{assumption}{Assumption}
\newtheorem{proposition}[theorem]{Proposition}

\newtheorem{remark}[theorem]{Remark}
\newcommand{\var}{{\rm Var} }

\newcommand{\ind}{\mathbf{1}}

\newcommand{\cA}{\ensuremath{\mathcal A}} 
 
\newcommand{\cC}{\ensuremath{\mathcal C}} 
 
\newcommand{\cE}{\ensuremath{\mathcal E}} 
\newcommand{\cF}{\ensuremath{\mathcal F}} 
 
\newcommand{\cH}{\ensuremath{\mathcal H}}

\newcommand{\cL}{\ensuremath{\mathcal L}} 
 
\newcommand{\cN}{\ensuremath{\mathcal N}} 
 
\newcommand{\cP}{\ensuremath{\mathcal P}}

\newcommand{\cS}{\ensuremath{\mathcal S}} 
\newcommand{\cT}{\ensuremath{\mathcal T}}

\newcommand{\cX}{\ensuremath{\mathcal X}}

\newcommand{\w}{\mathbf{w}}

\newcommand{\E}{\ensuremath{\mathbb{E}}}

\newcommand{\N}{\ensuremath{\mathbb{N}}}

\newcommand{\R}{\ensuremath{\mathbb{R}}}
\renewcommand{\P}{\ensuremath{\mathbb{P}}}
\newcommand{\pl}{\ensuremath{\left\langle}}
\newcommand{\pr}{\ensuremath{\right\rangle}}

\newcommand{\nash}{t_{\text{\normalfont Nash}}}

\newcommand{\dtwo}{d'}

\newcommand{\rel}{\text{\normalfont rel}}
\newcommand{\cxy}{\ensuremath{c_{xy}}}

\newcommand{\bin}{\ensuremath{{\rm Bin}}}

\newcommand{\avg}{\text{\normalfont Avg}}
\newcommand{\bx}{\mathbf{x}}
\newcommand{\by}{\mathbf{y}}

\newcommand{\TV}{\text{\normalfont TV}}

\newcommand{\dd}{\text{\normalfont d}}
\newcommand{\vertiii}[1]{{\left\vert\kern-0.25ex\left\vert\kern-0.25ex\left\vert #1 
		\right\vert\kern-0.25ex\right\vert\kern-0.25ex\right\vert}}
\newcommand{\Ker}{\text{\normalfont Ker}}
\newcommand{\Img}{\text{\normalfont Im}}
\newcommand{\ann}{{\mathfrak a}}
\newcommand{\acr}{{\mathfrak a}^\dagger}
\newcommand{\mx}{\text{\normalfont mix}}

\newcommand{\pix}{\pi(x)}
\newcommand{\piy}{\pi(y)}

\newcommand{\xix}{{\xi(x)}}
\newcommand{\xiy}{{\xi(y)}}
\newcommand{\xiz}{{\xi(z)}}
\newcommand{\xip}{{\xi'}}

\newcommand{\Sym}{\text{\normalfont Sym}}
\newcommand{\dtv}{\mathbf{d}_k}
\newcommand{\identity}{\mathds{1}}
\newcommand{\tcdsz}{t_{\rm CDSZ}}
\newcommand{\tcdszw}{t_{\rm CDSZ,\w}}
\def\({\left(}
\def\){\right)}
\def\[{\left[}
\def\]{\right]}
\def\gap{\mathop{\rm gap}\nolimits}

\makeatletter
\newcommand*\notwithin[2]{%
	\@removefromreset{#1}{#2}%
}
\makeatother

\newacro{TV}{total variation}
\newacro{SSEP}{Symmetric Simple Exclusion process}
\newacro{ASEP}{Asymmetric Exclusion process}
\newacro{ZRP}{Zero Range process}
\endlocaldefs

\allowdisplaybreaks

\begin{document}

\begin{frontmatter}
\title{Mixing of the Averaging process and its discrete dual\\ on finite-dimensional geometries}
%\title{A sample article title with some additional note\thanksref{t1}}
\runtitle{Mixing of the Averaging process and its discrete dual}
%\thankstext{T1}{A sample additional note to the title.}

\begin{aug}
%%%%%%%%%%%%%%%%%%%%%%%%%%%%%%%%%%%%%%%%%%%%%%
%%Only one address is permitted per author. %%
%%Only division, organization and e-mail is %%
%%included in the address.                  %%
%%Additional information can be included in %%
%%the Acknowledgments section if necessary. %%
%%%%%%%%%%%%%%%%%%%%%%%%%%%%%%%%%%%%%%%%%%%%%%
\author[A]{\fnms{Matteo} \snm{Quattropani}\ead[label=e1]{m.quattropani@math.leidenuniv.nl}},
\author[B]{\fnms{Federico} \snm{Sau}\ead[label=e2]{federico.sau@ist.ac.at}}
%%%%%%%%%%%%%%%%%%%%%%%%%%%%%%%%%%%%%%%%%%%%%%
%% Addresses                                %%
%%%%%%%%%%%%%%%%%%%%%%%%%%%%%%%%%%%%%%%%%%%%%%
\address[A]{Mathematical Institute,
Leiden University,
\printead{e1}}

\address[B]{Institute of Science and Technology Austria, ISTA,
\printead{e2}}
\end{aug}

\begin{abstract}
We analyze the $L^1$-mixing of a generalization of the 
Averaging process introduced by Aldous \cite{aldous2011finite}. The process takes place on a growing sequence of graphs which we assume to be finite-dimensional, in the sense that the random walk on those geometries satisfies a family of Nash inequalities. As a byproduct of our analysis, we provide a complete picture of the total variation mixing of a discrete dual of the Averaging process, which we call Binomial Splitting process. A single particle of this process is essentially the random walk on the underlying graph. When several particles evolve together, they interact by synchronizing their jumps when placed on neighboring sites. We show that, given $k$ the number of particles and $n$ the (growing) size of the underlying graph, the system exhibits cutoff in total variation if $k\to\infty$ and $k=O(n^2)$. Finally, we exploit the duality between the two processes to show that the Binomial Splitting process satisfies a version of Aldous' spectral gap identity, namely, the relaxation time of the process is independent of the number of particles.
\end{abstract}



\end{frontmatter}
%%%%%%%%%%%%%%%%%%%%%%%%%%%%%%%%%%%%%%%%%%%%%%
%% Please use \tableofcontents for articles %%
%% with 50 pages and more                   %%
%%%%%%%%%%%%%%%%%%%%%%%%%%%%%%%%%%%%%%%%%%%%%%
%\tableofcontents

\section{Introduction}
Introduced in a series of lectures and expository articles by Aldous (\cite{aldous2011finite,aldous_lecture_2012,aldous2013interacting}), the Averaging process is a Markovian model of mass redistribution among nearest-neighboring sites of a graph. Informally, the Averaging process may be described as follows: after initially assigning some real values to each site, at  exponentially-distributed times   neighboring sites are  selected and, then, split \emph{equally} among themselves their total mass.

Originally proposed as a basic mathematical model for social dynamics, the Averaging process naturally fits into the growing class of \emph{opinion exchange models} (see, e.g., the recent survey \cite{mossel_opinion2017}). In this context sites and edges represent agents together with their connections, while  the sites' values  measure their  opinions. Such stochastic models are employed with the scope of quantitatively studying, e.g.,  the conditions and timescales leading to consensus as well as the role of the underlying graph topology in this. 

The Averaging process is also closely related
to a large number of models of mass redistribution from statistical physics, economics
and computer science, see, e.g., \cite[\S 1.1]{chatterjee2020phase}. Moreover, this and other Markovian models with a continuous state space serve as constituent examples to extend the  geometric theory of discrete Markov chains (see, e.g., \cite{saloff1997lectures,levin2017markov,maas_gradient_2011}) to the continuous setting. Among the  recent  works in this direction, we mention \cite{smith_analysis_2013,smith_gibbs_2014,pillai2018mixing,caputo_mixing_2019,caputo2020spectral, banerjee2020rates,chatterjee2020phase} as those being concerned  with the study of  convergence rates and spectral gap's identities for continuous mass redistribution models.

Despite the several analogies with statistical mechanics
models, the Averaging process shares the distinguishing feature of reaching \emph{equilibrium} at a \emph{single absorbing state} with most  examples in opinion exchange dynamics. This singularity, together with the
continuous nature of the state space and the heterogeneity of the underlying social network,
is what makes these Markovian models mathematically interesting and challenging.

In this paper we enhance the analysis of the $L^1$-transportation metric mixing of an inhomogeneous ``unfair'' generalization of the Averaging process. The setting is that of large undirected graphs satisfying finite-dimensional Nash inequalities; the latter encompasses, for instance, the segment, the circle, the discrete $d$-dimensional torus and all discrete approximations of ``nice'' Euclidean domains (see \cref{sec:setting} below for further details and examples).
We analyze the distance-to-equilibrium at different scales. In particular, we show that the process gradually mixes on the scale $\Theta(1)$, while it exhibits an abrupt behavior when zooming in on finer scales.
Furthermore, the bounds that we obtain allow us to carry out a complete analysis of the mixing of a discrete analogue of the Averaging process, which we refer to as Binomial Splitting process. This model is an interacting particle system with a conservation law, and a significant part of this paper is devoted to the study of its spectral gap and \ac{TV} cutoff.

The last decade has registered important breakthroughs on the understanding of the sharp mixing behavior for conservative particle systems.  Among the most influential  works, we mention those by Caputo et al.\ \cite{caputo_proof_2010} on the Aldous' spectral gap conjecture, and by Lacoin \cite{lacoin2016mixing} on the \ac{TV} cutoff on one-dimensional domains, both for the \ac{SSEP}. More recently, versions of Aldous' spectral gap identity have been shown for other models, e.g.,  the \ac{ZRP} \cite{hermon_version_2019}, Beta and Gibbs samplers on the segment \cite{caputo_mixing_2019,caputo2020spectral}, Wright-Fisher and Fleming-Viot processes and multi-allelic Moran models \cite{shimakura_equations1977,griffiths_lambda2014,ren_wang_spectral2020,corujo2020spectrum}. Similarly, we refer, among others, to \cite{lacoin_leblond_cutoff2011,jonasson_mixing2012,labbe_lacoin_cutoff2019,labbe_lacoin_mixing2020,schmid_mixing2019,schmid2021mixing,merle_salez_2019,hermon2020cutoff} for recent developments on the cutoff for \ac{SSEP}, its asymmetric variants and \ac{ZRP}.

The Binomial Splitting process shares some features with some symmetric particle systems mentioned above. Among these, the presence of dualities and intertwinings play a prominent role in our work.
By means of dual descriptions, on the one hand we derive sharp upper bounds for the Averaging process from properties of a-few-particle Binomial Splitting; on the other hand, we also deduce results on the many-particle Binomial Splitting through the analysis of the Averaging dynamics.

\section{Models \& main results}\label{sec:models-results}
We devote this section to the rigorous description of the Markovian models and to the corresponding main results.
The underlying common geometry for the two models is represented by a weighted, connected and undirected graph $G=(V,E,\(\cxy\)_{xy\in E})$ and site-weights $\pi=\(\pix\)_{x\in V}$, a \emph{non-degenerate} probability measure on $V$, i.e., $\pix>0$ for all $x\in V$. Moreover, for all $p\in[1,\infty]$, we let $L^p(V,\pi)$ denote the Banach space of functions $\psi:V\to \R$ endowed with the norm $\|\cdot\|_p$ ($\left\|\psi\right\|_p^p\coloneqq \sum_{x\in V}\pix\left|\psi(x)\right|^p$ for $p\in[1,\infty)$ and $\left\|\psi\right\|_\infty\coloneqq \max_{x\in V}\left|\psi(x)\right|$ for $p=\infty$). For the case $p=2$, we use the shorthand notations $\cH$ for $L^2(V,\pi)$ and $\pl \cdot,\cdot\pr$ for the corresponding inner product.

\subsection{Binomial  Splitting process}\label{sec:binomial_splitting}
The Binomial Splitting process is a natural discrete analogue of the Averaging process, in which pairs of sites split particles rather than mass, according to a Binomial distribution rather than deterministically. More precisely, for each $k\in\N$, the Binomial Splitting process with $k$ particles, $\bin(k)$, is the continuous-time Markov process $(\xi_t)_{t\ge 0}$ on the finite configuration space \begin{equation}\Omega_k\coloneqq \left\{\xi\in \N_0^V \:\bigg\rvert\: \sum_{x\in V}\xix=k \right\},
\end{equation}
whose evolution is described by the infinitesimal generator
\begin{equation}\label{eq:gen-unlabled}
\cL^{\bin(k)} f = \sum_{xy\in E}\cxy \(\cP^{\bin(k)}_{xy}-\identity\)f,\qquad  f:\Omega_k\to\R.
\end{equation}
In the above formula, 
\begin{equation}\label{eq:pbin}
\cP^{\bin(k)}_{xy} f(\xi)\coloneqq \E_{\Xi^{xy}_\xi}\[f\],
\end{equation}
where the $\Omega_k$-valued r.v.\ $\Xi^{xy}_\xi$ is defined as
\begin{equation}\label{eq:def-Xi}
\Xi^{xy}_\xi(z)\coloneqq\begin{cases}
Y&\text{if }z=x\\
\xix+\xiy-Y&\text{if }z=y\\
\xiz&\text{otherwise}
\end{cases}
\end{equation}
with
\begin{equation}
 Y\sim \text{Binomial}\(\xix+\xiy,\frac{\pix}{\pix+\piy}\).
\end{equation}
Here and in what follows, if not stated otherwise, $\E_\nu$ denotes the expectation with respect to the probability law $\nu$; if $\nu=\text{Law}(X)$, for some r.v. $X$, we abbreviate $\E_{\text{Law}(X)}$ by $\E_{X}$. In \cref{eq:def-Xi}, notice the symmetry of the r.v. $\Xi_{\xi}^{xy}$ with respect to the sites $x,y\in V$. 

As a straightforward detailed balance computation shows, the unique invariant and reversible measure for the process $\(\xi_t \)_{t\ge0}$ is Multinomial with parameters $(k,\pi)$, which is denoted by
\begin{equation*}
	\mu_{k,\pi}(\xi)=k!\:\prod_{x\in V}\frac{\pi(x)^\xix}{\xix!},
\end{equation*}
and the Binomial Splitting dynamics acts as a local instantaneous thermalization at pairs of sites with rates $\(\cxy \)_{xy\in E}$. It is worth to observe that, while the Binomial redistribution rule corresponds to an i.i.d.\ sampling of the particles' new locations, interaction enters only when forcing them to jump simultaneously. A quantitative analysis of the effects of such a weak dependence lies at the core of our work.

Notice that if the system consists of a single particle, it can be equivalently described through its position $(X_t)_{t\ge0}$ on $V$,  its reversible measure coinciding with $\pi$; we call $\P^{\bin(1)}$ the law of such a Markov chain.

\subsection{Aldous' spectral gap identity}\label{sec:gap}
The first result we present is a spectral gap identity for the Binomial Splitting process. In words, we show that the $k$-particle system's spectral gap coincides with the spectral gap of the single-particle system on any graph.

While the literature on spectral gap estimates is too vast to be all mentioned here, the understanding of the exact correspondence between many- and  single-particle systems' spectral gaps is limited, besides the trivial case of independent particles, to a few more examples \cite{caputo_proof_2010,caputo2020spectral,caputo_kac2008,caputo_mixing_2019,hermon_version_2019,corujo2020spectrum}. The next theorem adds a further model to the previous list.

Let us first introduce some terminology. Due to reversibility of the Multinomial$(k,\pi)$, denoted below by $\mu_{k,\pi}$, with respect to the $\bin(k)$ dynamics, the Rayleigh quotient representation for the spectral gap holds, i.e., the smallest positive eigenvalue of $-\cL^{\bin(k)}$ can be equivalently defined by
\begin{equation}\label{eq:def-spectral-gap}
\gap_k\coloneqq \inf_{f\neq 0:\:\E_{\mu_{k,\pi}}[f]=0}\frac{\E_{\mu_{k,\pi}}\[f(-\cL^{\bin(k)}f) \]}{\E_{\mu_{k,\pi}}\[f^2 \]}.
\end{equation}

\begin{theorem}\label{th:gap}
	For all graphs $G=\big(V,E,\(\cxy\)_{xy\in E}\big)$, non-degenerate site-weights $(\pix)_{x\in V}$ and $k\in \N$, 
	\begin{equation}\label{eq:def-gap}
	\gap_k=\gap_1\eqqcolon \gap.
	\end{equation}
\end{theorem}
The proof of this result is deferred to \cref{sec:gap_proof} below.

\subsection{Asymptotic framework and examples}\label{sec:setting}

In contrast with the spectral gap result in \cref{th:gap}, which holds for every fixed (weighted) graph $G$ and non-degenerate site-weights $\pi$, all results presented in the forthcoming subsections  are framed---in analogy with most of the literature on Markov chains' mixing times---in an asymptotic setting. More in detail, we will consider a growing sequence of weighted graphs $G_n=(V_n,E_n,\(\cxy\)_{xy\in E_n})$ with corresponding site weights $(\pix)_{x\in V_n}$. The size of the vertex set $|V_n|=n$ will play the role of the diverging parameter and all the asymptotic notation will refer to the limit $n\to\infty$; moreover,  the dependence on $n$ will be usually omitted. In what follows we use the usual  Landau asymptotic notation: given two non-negative sequences $(f_n)_{n\in\N}$, $(g_n)_{n\in\N}$, we write
\begin{align*}
	f_n=O(g_n)&\quad\iff\quad \limsup_{n\to\infty}\frac{f_n}{g_n}<\infty,\\
	f_n=o(g_n)&\quad\iff\quad \limsup_{n\to\infty}\frac{f_n}{g_n}=0,\\
	f_n=\Omega(g_n)&\quad\iff\quad \liminf_{n\to\infty}\frac{f_n}{g_n}>0,\\
	f_n=\omega(g_n)&\quad\iff\quad \liminf_{n\to\infty}\frac{f_n}{g_n}=\infty,\\
	f_n=\Theta(g_n)&\quad\iff\quad 0<\liminf_{n\to\infty}\frac{f_n}{g_n}\le \limsup_{n\to\infty}\frac{f_n}{g_n}<\infty.%,\\
%	f_n\sim g_n&\quad\iff\quad \lim_{n\to\infty}\frac{f_n}{g_n}=1.
\end{align*}

\

The main assumption that we require is that the growing graphs are, roughly speaking, finite-dimensional. Nash inequalities will be the analytical tool which encodes the finite-dimensionality of our geometries. One of the strengths of such integral inequalities is that they imply the ultracontractivity of the random walk's Markov semigroup or, in other words, pointwise heat kernel \textquotedblleft on-diagonal\textquotedblright\ upper bounds. This property, combined with tensorization and comparison results, will turn out to be extremely useful also in our setting of many interacting particles.

Originally developed in the context of parabolic PDEs by Nash (\cite{nash_continuity1958}, see also \cite{fabes_new1986}),  this class of functional inequalities has been first exploited in the context of finite-state Markov chains in \cite{diaconis1996nash}.  In the reversible setting, \cite{carlen_upper_1986} established the equivalence between Nash inequalities and ultracontractivity (see also \cite{coulhon_ultracontractivity1996}). Several other conditions are known to imply Nash inequalities, for instance: isoperimetry (see, e.g., \cite[Section 3.3.2]{saloff1997lectures}); moderate growth conditions in combination with local Poincar\'e inequalities (\cite[Theorem 5.2]{diaconis1996nash}); upper and lower Gaussian-like heat kernel bounds or parabolic Harnack inequalities (\cite[Section 5.2]{barlow_random_2017}). Moreover, as most of these conditions, Nash inequalities transfer from infinite graphs to sequences of uniformly roughly isometric finite graphs, see \cite[Corollary 2.10 and Proposition 3.1]{dembo_cutoff2018}.

\

Let us now formally present our assumption concerning Nash inequality for the single-particle system, i.e., $\bin(1)$ with law $\P^{\bin(1)}$ as defined in the end of \cref{sec:binomial_splitting}. For all $n \in \N$, we say that $\bin(1)$ satisfies a \emph{Nash inequality with (positive) constants $d=d(n)$,  $\nash=\nash(n)$ and $T=T(n)$} if
\begin{equation}\label{eq:assumption_nash1_bin1-1}
\left\|\psi \right\|_2^{2\left(1+\frac{2}{d} \right)}\leq \nash\(\cE_{\bin(1)}(\psi)+T^{-1}\left\|\psi\right\|_2^2\) \left\|\psi \right\|_1^{\frac{4}{d}},\qquad \psi \in  \cH,
\end{equation}
holds, 
where $\cE_{\bin(1)}$ denotes the Dirichlet form of the single-particle system, viz.\ 
\begin{equation}\label{eq:def-dirichlet-bin1}
\cE_{\bin(1)}(\psi)\coloneqq \sum_{xy\in E}\cxy\, \frac{\pix\piy}{\pix+\piy}\big(\psi(x)-\psi(y) \big)^2.
\end{equation}
For every fixed $n\in\N$, a Nash inequality for \emph{some} positive constants $(d,\nash,T)$ always holds; hence, in an asymptotic setting the point lies in finding a sequence of constants $(d,\nash,T)_{n\in\N}$ with a suitable asymptotic behavior as,  for instance, in the forthcoming \cref{assumption:nash}.
Furthermore, as mentioned in the paragraph above, the integral inequality in \cref{eq:assumption_nash1_bin1-1} can be translated into a pointwise upper bound for the heat kernel of the single-particle process. More precisely, calling
\begin{equation}\label{eq:def-hx}
h_t^x(y)\coloneqq\frac{\P^{\bin(1)}(X_t=y\mid X_0=x)}{\piy} ,\qquad x,y\in V,\:t\ge0 ,
\end{equation}
\cref{eq:assumption_nash1_bin1-1} implies that (see  \cite[Theorem 2.3.4]{saloff1997lectures})
\begin{equation}\label{eq:nash-decay}
\max_{x,y\in V} h_t^x(y)\le e \(\frac{d\nash }{2 t}\)^{\frac{d}{2}},\qquad t\leq T.
\end{equation}
Notice that, for the purpose  of deriving \cref{eq:nash-decay} above, it suffices to check Nash inequality in \cref{eq:assumption_nash1_bin1-1}  for all $\psi \in\mathscr{P}_\pi \subseteq \cH$, the subset of probability densities with respect to $\pi$. 
Moreover, as the above inequality shows,  in an asymptotic framework Nash inequalities  as in \cref{eq:assumption_nash1_bin1-1} are most useful  when  $d$ is small, $T=\Theta(\nash)$ and $\nash$ is at most of the same order of the \emph{relaxation time}, i.e.,
\begin{equation}
t_\rel\coloneqq \gap^{-1}.
\end{equation}
In this case,  $\nash$ plays the role of burn-in time around  which all $L^p$-norms of the Markov chain's probability density become uniformly bounded and $\bin(1)$ mixes gradually at times $\Theta(t_\rel)$.  
This discussion motivates  the following assumption  on the underlying geometries satisfying a \textquotedblleft good\textquotedblright\ family of Nash inequalities. 
\begin{assumption}[Finite-dimensional geometries]\label{assumption:nash} For all $n \in \N$, we assume that Nash inequality for $\bin(1)$ holds with positive constants $d=d(n)$, $T=T(n)$ and $\nash=\nash(n)$ satisfying 
	\begin{equation}
		d=O(1),\qquad  \nash=O(t_\rel),\qquad T=\Theta(\nash).
	\end{equation}
%	\begin{equation}
%	\limsup_{n\to\infty}d\leq c_{\rm dim},\qquad \limsup_{n\to \infty}\frac{\nash}{T}\leq c_{\rm Nash},
%	\end{equation}
%	and
%	\begin{equation}\label{eq:hp-tnash-trel}
%	\limsup_{n\to\infty} \frac{\nash}{t_\rel}\le c_{\rm ratio},
%	\end{equation}
%	for   some  $c_{\rm dim}, c_{\rm Nash}$ and $c_{\rm ratio}\in[0,\infty)$.
\end{assumption}
Notice that \cref{assumption:nash} implies that there exist  $c_{\rm dim}$ and $c_{\rm ratio}\in(0,\infty)$ such that
	\begin{equation}\label{eq:hp-tnash-trel}
		\sup_{n\in \N}d\le c_{\rm dim}\qquad \text{and}\qquad \sup_{n\in \N}\frac{\nash}{t_\rel} \le c_{\rm ratio}.
	\end{equation}
Moreover, the assumption of non-degeneracy of the site-weights used in the proof of \cref{th:gap} is replaced in the current asymptotic setting by the following uniform non-degeneracy assumption.
\begin{assumption}[Uniformly elliptic site-weights]\label{assumption:unif_ellipticity}
	We assume that the sequence of site-weights satisfies
	\begin{equation}
	\sup_{n\in \N}\frac{\max_{x\in V}\pix}{\min_{y\in V}\piy}\le  c_{\rm ell}
	\end{equation}
	for some   $c_{\rm ell}\in[1,\infty)$.
\end{assumption}
\begin{remark}
	If \cref{assumption:nash} holds for some uniformly elliptic probability distribution $\pi$ on $V$, then, just by simple comparison of norms and Dirichlet forms in \cref{eq:assumption_nash1_bin1-1}, it holds for any other  uniformly elliptic  $\pi'$, with $\nash'=\Theta(\nash)$, $t'_\rel=\Theta(t_\rel)$, etc.\ . 
\end{remark}

We conclude this section with a  list of examples of sequences of graphs satisfying  \cref{assumption:nash} for every uniformly elliptic $\pi$ as in \cref{assumption:unif_ellipticity}  (in  all the 	following examples, we omit to remark that $T=\Theta(\nash)$ holds):

\begin{itemize}
	\item 
	Lattice discretizations of \emph{$m$-dimensional tori} and \emph{Euclidean bounded Lipschitz domains} $\cA \subset \R^m$, $m\geq 1$, where $G=G_n$ is a  lattice approximation of $\cA$  with nearest-neighbor jumps and uniformly elliptic conductances, see \cite{chen2017hydrodynamic}: 
	\begin{equation}
	d=m,\qquad
	t_\rel = \Theta(n^{2/d} ),\qquad \nash= \Theta(n^{2/d}).
	\end{equation}
	As   special	 cases, one recovers the circle and the segment  with  uniformly elliptic conductances.
	Notice that, by a  comparison argument, the restriction to  nearest-neighbor jumps may be relaxed up to include finite-range jumps. Moreover,  instances of random walks on finite groups with moderate growth also belong to this same class, see \cite{diaconis1994moderate, diaconis1996nash}	.
	\item Random walks on large finite $m$-dimensional  boxes on the \emph{supercritical percolation cluster}, see \cite[Eq.\ (6)  \& Theorem 1.3]{mathieu_remy_isoperimetry2004}:
	\begin{equation}
	d=m+o(1),\qquad t_\rel = \Theta(n^{2/d}),\qquad \nash=\Theta(n^{2/d}).
	\end{equation}	 	
	\item The \textquotedblleft $n$-dog\textquotedblright\ from \cite[Examples 3.3.2 \& 3.3.5]{saloff1997lectures}, namely the lattice discretization of two boxes of $\R^2$ intersecting only in one corner:
	\begin{equation}
	d=2,\qquad t_\rel=\Theta(n\log n),\qquad \nash=\Theta(n).
	\end{equation} 
	\item Finite fractal graphs  with bounded \emph{walk-dimension} $d_w> 2$ and \emph{volume growth exponent} $d_f>0$:
	\begin{equation}
	d=2d_f/d_w,\qquad t_\rel=\Theta(n^{2/d}),\qquad \nash=\Theta(n^{2/d}).	
	\end{equation} Heat kernel estimates for several examples of this kind have been extensively studied  on \emph{infinite}  graphs, see, e.g., \cite{jones_transition1996,kigami_analysis2001,barlow_random_2017};  \emph{finite}-graph examples, including  Sierpi\'nski gasket and carpet graphs,   are thoroughly discussed in \cite{dembo_cutoff2018}.   
\end{itemize} 
It is rather straightforward to construct  examples of geometries for which our \cref{assumption:nash} does \emph{not} hold. For instance, any sequence constructed from the list   above with diverging $d$  (e.g., $m$-dimensional Euclidean boxes with $m=m_n\to \infty$) belongs to this class. Another example is that of the (homogeneous) complete graphs with $n$ vertices as in \cref{th:CDSZ20} below; as a simple computation shows, $t_\rel =\frac{2}{n}\to 0$, while, for all possible choices of $d=O(1)$ and $T=\Theta(\nash)$, $\nash\to\infty$.

\subsection{Mixing 	 of the Binomial Splitting}\label{sec:results_bin}
This section is devoted to the presentation of the results concerning the \ac{TV} mixing of the Binomial Splitting process. For all $k \in \N$, we let $(\cS^{\bin(k)}_t)_{t\ge0}$ denote the Markov semigroup associated to the generator $\cL^{\bin(k)}$ in \cref{eq:gen-unlabled} and, for every initial distribution $\nu$ on $\Omega_k$, we adopt the matrix notation $\nu \cS^{\bin(k)}_t$ to refer to the corresponding distribution at time $t\ge 0$. 

Recall the definition of \ac{TV} distance to equilibrium at time $t\ge0$ when starting from some probability distribution $\nu$ over $\Omega_k$
\begin{equation}\label{eq:def-tv-nu}
\dtv^{(\nu)}(t)\coloneqq  \left\|\nu\cS^{\bin(k)}_t-\mu_{k,\pi} \right\|_{\TV}=\sup_{A\subset \Omega_k }\big|\nu\cS^{\bin(k)}_t(A)-\mu_{k,\pi}(A) \big|,
\end{equation}
where $\mu_{k,\pi}$ is the Multinomial distribution of parameters $(k,\pi)$.

In \cref{th:cutoff} below, we show that, in the asymptotic setting of \cref{sec:setting}, the worst-case \ac{TV} distance exhibits the so-called \emph{cutoff phenomenon}. Discovered in the 80's by Aldous and Diaconis, \cite{aldous1986shuffling}, the expression \emph{cutoff} refers, in the context of Markov chains, to an abrupt convergence to equilibrium measured with a given distance. In  recent years, several systems of interacting particles have been shown to exhibit cutoff in \ac{TV} distance, see, e.g., \cite{lacoin2016mixing,caputo_mixing_2019,caputo2020spectral,hermon2020cutoff,lubetzky_sly_2014,bufetov2020cutoff,levin_glauber2010} and references therein.

In analogy with this literature, we consider the \emph{worst-case} mixing, namely, we will take the supremum of the quantity in  \cref{eq:def-tv-nu} over the set of initial distributions which, in turn, is equivalent to take the maximum over the set $\Omega_k$. For this reason, letting $\delta_\xi$ denote the Dirac' distribution at $\xi\in\Omega_k$, we define
\begin{equation}\label{eq:def-tv-worst}
\dtv(t)\coloneqq\sup_{\xi\in\Omega_k}\dtv^{(\delta_\xi)}(t).
\end{equation}
We will show that, when $k\to\infty$,  cutoff for the worst-case \ac{TV} distance in \cref{eq:def-tv-worst} occurs around the time
\begin{equation}\label{def:tmix}
t_\mx \coloneqq\frac{t_\rel}{2}\log(k),
\end{equation}
thus, rightfully referred to as \emph{the} mixing time. Moreover, the quantity in \cref{eq:def-tv-worst} is bounded away from $0$ and $1$ in a window of size $\Theta(t_\rel)$ around $t_\mx$. In order to quantify the latter statement, for every $C>0$ we will consider the quantity
\begin{equation}\label{def:tw}
t_\w(C)\coloneqq Ct_\rel,
\end{equation}
which will play the role of the so-called \emph{cutoff window} in the forthcoming \cref{th:cutoff}. In addition, we further define
\begin{equation}\label{eq:def-t+t-}
t^\pm(C)\coloneqq t_\mx\pm t_\w(C).
\end{equation}

\begin{theorem}[Cutoff]\label{th:cutoff}
	Consider a sequence of graphs and site-weights such that \cref{assumption:nash,assumption:unif_ellipticity} hold. Let us further assume that the total number of particles $k=k_n$ is such that $k\to\infty$ and there exists some $c_{\rm vol}\in[0,\infty)$ for which
	\begin{equation}\label{eq:hp-k}
	\limsup_{n\to\infty} \frac{k}{n^2}\le c_{\rm vol}. 
	\end{equation}
	Then, for all  $\delta\in(0,1)$ there exists some $C=C(\delta)>0$ such that
	\begin{equation}
	\limsup_{n\to\infty} \dtv(t^+(C))\le \delta,\qquad
	\liminf_{n\to\infty} \dtv(t^-(C))\ge 1-\delta
	\end{equation}
	hold.
	
\end{theorem}
The assumption $k\to\infty$ is necessary for the validity of \cref{th:cutoff}. Indeed, when $k=1$, the so-called \emph{product condition} (see, e.g., \cite[Proposition 18.4]{levin2017markov}) together with \cref{eq:nash-decay} impose that in order for the cutoff to occur we need $\nash \gg t_\rel$, which contrasts our \cref{assumption:nash}. The next proposition shows that, in our setting, the absence of cutoff holds for any sequence $k=O(1)$.
\begin{proposition}\label{pr:no-cutoff}
	In the same setting of \cref{th:cutoff}, if $k=k_n=O(1)$, then, there exist $a,b>0$ independent of $n$ such that
	\begin{equation}
 e^{-\frac{t}{t_\rel}}\le \dtv(t )\le a e^{-\frac{t}{t_\rel}},\qquad t\ge b t_\rel.
	\end{equation}
\end{proposition}

\begin{remark}[High-density regime]\label{rem:high-density}
	In the regime in which $k\to\infty$ and $k=O(n^2)$ the cutoff time $t_\mx$ in \cref{def:tmix} coincides with the cutoff time for other recently studied  symmetric interacting systems (mostly in 1D), e.g., \cite{caputo2020spectral,caputo_mixing_2019,lacoin2016mixing}. While the assumption $k\to\infty$ is necessary for the validity of \cref{th:cutoff} (cf.\ \cref{pr:no-cutoff}), 
	this is not clear for the requirement  $k=O(n^2)$. In fact,  our techniques break down when dropping  that assumption.
	Hence,  determining the emergence of the cutoff phenomenon for the $\bin(k)$ in the \emph{high-density regime}, that is, when $k=\omega(n^2)$, remains an open problem. Nonetheless, as we will show in \cref{th:precutoff-bin} below, in this regime a timescale $\Theta(t_\rel \log(k))$ is still sufficient for the system to be well-mixed.
\end{remark}

The conclusions in \cref{th:cutoff} and \cref{pr:no-cutoff}---whose proofs are postponed to \cref{sec:proof_cutoff_bin} below---are drawn from   mixing results for the averaging  process, presented in the next two sections.

\subsection{Averaging process}
We start by a rigorous definition of the ``unfair'' Averaging process in which sites redistribute their mass proportionally to prescribed site-weights. Since the Averaging dynamics conserves the total mass and it is invariant under dilation, with no loss of generality we will assume that the initial configuration is some $\eta\in\Delta$, where $\Delta$ is the set of probability distributions over $V$. More precisely, given a non-degenerate $\pi\in\Delta$, the Averaging process is the Markov process $(\eta_t)_{t\ge0}$ with state space $\Delta$ and infinitesimal generator
\begin{equation}\label{eq:def-Lavg}
\cL^{\avg}f\coloneqq\sum_{xy\in E}\cxy\(\cP^\avg_{xy}-\identity\)f,\qquad f:\Delta\to\R
\end{equation}
where
\begin{equation}\label{eq:pavg}
\cP_{xy}^\avg f(\eta)\coloneqq f(\eta^{xy}),\qquad \eta^{xy}(z)\coloneqq\begin{cases}
\frac{\pix}{\pix+\piy}\(\eta(x)+\eta(y)\)&\text{if }z=x,\\
\frac{\piy}{\pix+\piy}\(\eta(x)+\eta(y)\)&\text{if }z=y,\\
\eta(z)&\text{otherwhise}.
\end{cases}
\end{equation}
Let $\cC(\Delta)$ be the Banach space of continuous functions on the compact metric space $(\Delta,\|\cdot\|_2)$ endowed with the supremum norm. Then, it is easy to check that the operator $\cL^\avg$ is a bounded linear operator on $\cC(\Delta)$ and that generates a Feller Markov contraction semigroup $\big(\cS^\avg_t\big)_{t\ge 0}$ on  the same space. All throughout, $\P^\avg_\nu$ and $\E^\avg_\nu$ denote the law and corresponding expectation of the Averaging process distributed at time $t=0$ according to $\nu$;  when $\nu$ is a Dirac at $\eta$, we simply write $\P^\avg_\eta$ and $\E^\avg_\eta$. An analogous  notation will be adopted for the $k$-particle Binomial Splitting.

\subsection{Mixing of the Averaging}\label{sec:avg-results}
As it can be read off the definition in \cref{eq:def-Lavg}, $\pi \in \Delta$ is the unique absorbing point of the Averaging dynamics. Moreover, although the sequence of local thermalizations is random, the deterministic nature of the mass redistributions prevents any other point $\eta \in \Delta$ to be visited more than once, breaking down any reversibility. In this context, quantifying the convergence to stationarity entails a sensible choice of distance to equilibrium. Indeed, at any finite time there is a positive probability that only a fraction of edges has been updated. Thus, the distribution of the Averaging process is singular with respect to the unique stationary measure, the Dirac $\delta_\pi$. Hence, convergence to stationarity cannot occur in, e.g., \ac{TV} distance, i.e., $\|{\rm Law}(\eta_t)- \delta_{\pi}\|_{\rm TV}\not\to0$ as $t\to \infty$. Therefore, as in other recent works on the Averaging process and related mass redistribution models (see, e.g.,	 \cite{aldous_lecture_2012,chatterjee2020phase,banerjee2020rates}), we will employ Wasserstein-type of distances. More precisely, we will adopt $L^p$-transportation metrics, namely, for all $p \in [1,2]$,
\begin{equation}
\E^\avg_\eta\[\left\|\frac{\eta_t}{\pi}-1\right\|_p\],\qquad   \eta \in \Delta,\ t \ge 0.
\end{equation}

These metrics set the ground for a quantitative comparison between the Averaging and its \textquotedblleft noiseless\textquotedblright\ counterpart. Indeed, as it will become clear with the statement on the duality relations in \cref{pr:dualities} below, the Averaging process $\(\eta_t\)_{t\ge 0}\subseteq \Delta$ decomposes into two components: a deterministic part $\(\pi_t\)_{t \ge 0}\subseteq \Delta$ corresponding to the law of the single-particle Binomial Splitting system, and a \textquotedblleft noise\textquotedblright\ part $\(\cX_t\)_{t\ge 0}\subseteq \R^V$, such that
\begin{equation}
\eta_t=\pi_t+\cX_t,\qquad t \ge0,
\end{equation}
with $\cX_0=0$ and $\sum_{x\in V}\cX_t(x)=0$ a.s., $\E\[\cX_t\]=0$ for all $t \ge0$, as well as $\cX_t\to 0$ in law as $t \to \infty$.

In view of these considerations, it comes natural to compare rates of convergence for the Averaging in $L^p$-transportation distance and for $\bin(1)$ in $L^p$-distance. This comparison boils down to analyze the effect of the noise $\(\cX_t\)_{t\ge 0}$ at various scales, leading to possible mismatches in the behaviors of the two processes.
A first instance of this phenomenon is shown in the next proposition, which generalizes a result taken from  \cite{aldous_lecture_2012}
 to the inhomogeneous context, i.e., when $\pi$ is not uniform.

\begin{proposition}[{Cf.\ \cite[Proposition 2]{aldous_lecture_2012}}]\label{pr:aldous-lanoue}
	For all graphs $G=\big(V,E,\(\cxy\)_{xy\in E}\big)$ and non-degenerate site-weights $(\pix)_{x\in V}$,
	\begin{equation}
	\E_\eta^\avg\[ \left\|\frac{\eta_t}{\pi}-1 \right\|_{2}^2 \]\le e^{-\frac{t}{t_\rel}}\left\|\frac{\eta}{\pi}-1 \right\|_2^2,\qquad\eta\in\Delta,\quad t\ge 0.
	\end{equation} 
\end{proposition}
The above proposition is key to the proof of \cref{th:gap} and, for completeness, its proof is reported in \cref{sec:gap_proof} below. In words, \cref{pr:aldous-lanoue} above shows that  the contraction rate for the Averaging's $L^2$-transportation metrics is, in general, off by a factor $2$ from the standard $L^2$-contraction rate  prescribed by Poincar\'e inequality for $\bin(1)$. In fact, \cite[Corollary 2.2]{chatterjee2020phase} proves that the lack of this pre-factor  is exact in the specific context of the homogeneous complete graph.

The result in \cref{pr:aldous-lanoue} required no hypothesis other than the non-degeneracy of $\pi \in \Delta$. When turning to the analysis of the asymptotic behaviors, different assumptions on the underlying geometry may lead to dissimilar outcomes. To the best of our knowledge, {\cite{chatterjee2020phase}} is the only work so far establishing sharp results on the mixing of the Averaging process in an asymptotic setting.
As some of their findings directly relate to our results, we schematically report them below,
referring the interested reader to \cite{chatterjee2020phase} for further
details.

\begin{theorem}[{\cite[Theorems 1.1 \& 1.2]{chatterjee2020phase}}]\label{th:CDSZ20}
	Consider a sequence  of  growing complete graphs with $n$ vertices, unitary conductances and uniform site-weights. 
	Then, calling
	\begin{equation}
	\tcdsz\coloneqq\frac{1}{\log(2)}\frac{\log(n)}{n},\qquad
	\tcdszw(C)\coloneqq C\frac{\sqrt{\log(n)}}{n}
	\end{equation}
	and
	\begin{equation}
	\tcdsz^\pm(C)\coloneqq 	\tcdsz\pm 	\tcdszw(C),
	\end{equation}
	for all  $\delta\in(0,1)$ there exists some $C=C(\delta)>0$ such that
	\begin{equation}
	\begin{split}
	\limsup_{n\to\infty}\sup_{\eta\in\Delta}\E_\eta^\avg\[ \left\|\frac{\eta_{	\tcdsz^+(C)}}{\pi}-1 \right\|_{1} \]  &\le \delta,\\
	\liminf_{n\to\infty}\sup_{\eta\in\Delta}\E_\eta^\avg\[ \left\|\frac{\eta_{	\tcdsz^-(C)}}{\pi}-1 \right\|_{1} \]&\ge 2-\delta.
	\end{split}
	\end{equation}
\end{theorem}

The above theorem allows a  comparison  between the mixing behaviors  of the Averaging and of the $\bin(1)$ on the complete graph. Indeed, on the one hand,  the  $L^1$-Wasserstein distance to equilibrium for the Averaging  sharply drops to zero around times	$(\frac{1}{\log(2)}+o(1))\frac{\log(n)}{n}$; on the other hand, as a simple computation shows, the law of the  single $\bin$-particle  reaches equilibrium in \ac{TV} distance on a strictly shorter timescale, namely $t_\mx=\Theta(t_\rel)=\Theta(\frac{1}{n})$. In other words,  mixing of these two processes differ both qualitatively (i.e., abrupt vs.\ gradual) and quantitatively (i.e., on different timescales) on the \textquotedblleft infinite-dimensional\textquotedblright\ example of growing complete graphs.

A natural question is whether this disagreement occurs also  on  finite-dimensional geometries, namely those for which our \cref{assumption:nash} holds and, thus, the $\bin(1)$ mixes on a timescale $\Theta(t_\rel)$ without cutoff. As our next result shows, in such a framework, the mixing behaviors of the Averaging and of the $\bin(1)$ match, both occurring gradually at times $\Theta(t_\rel)$.

\begin{proposition}[No cutoff]\label{pr:no-cutoff-avg}
	Consider a sequence of graphs and site-weights such that \cref{assumption:nash} holds. Then,  there exist $a,b>0$ independent of $n$ such that 
	\begin{equation}\label{eq:no-cutoff-avg}
e^{-\frac{t}{t_\rel}}\le\E_\eta^\avg\[ \left\|\frac{\eta_{t}}{\pi}-1 \right\|_{p} \]\le ae^{-\frac{t}{t_\rel}},\qquad t\ge bt_\rel,
	\end{equation}
	for all $p\in[1,2]$.
\end{proposition}

Pushing further the analogy between Averaging dynamics and the \textquotedblleft noiseless\textquotedblright\ $\bin(1)$, we ask whether cutoff occurs on a \textquotedblleft finer scale\textquotedblright, i.e., when measuring the $L^p$-transportation metrics on scales of the order of $k^{-1/2}$, for some diverging $k=k_n\to \infty$. In the forthcoming theorem we prove that, in our asymptotic framework, the Averaging process abruptly mixes just like  the single particle on finer scales as long as $k=O(n^2)$. An  extension similar to that highlighted for the Binomial Splitting in \cref{rem:high-density} holds for the Averaging in the regime in which $k=\omega(n^2)$; see \cref{th:precutoff-avg} below for further details.
\begin{theorem}[Cutoff on a finer scale]\label{th:mixing-avg}
	Consider a sequence of graphs and site-weights such that \cref{assumption:nash,assumption:unif_ellipticity} hold. Fix also a sequence $k=k_n$ such that $k\to\infty$ and there exists some $c_{\rm vol}\in[0,\infty)$ for which \cref{eq:hp-k} holds. Then, for all  $\delta\in(0,1)$ there exists some $C=C(\delta)>0$ such that
	\begin{equation}\label{eq:cutoff-avg}
	\limsup_{n\to\infty} \sqrt{k}\: \sup_{\eta\in\Delta}\E_\eta^\avg\[\left\|\frac{\eta_{t^+}}{\pi}-1 \right\|_{p} \]\le\delta,\qquad
	\liminf_{n\to\infty} \sqrt{k}\: \sup_{\eta\in\Delta}\E_\eta^\avg\[\left\|\frac{\eta_{t^-}}{\pi}-1 \right\|_{p} \]\ge \frac{1}{\delta},
	\end{equation}
	for all $p\in[1,2]$, where $t^\pm(C)$ are defined as in \cref{eq:def-t+t-}.
\end{theorem}
As it will be shown in \cref{sec:proof_cutoff_bin} below, the results in \cref{pr:no-cutoff-avg,th:mixing-avg} will be instrumental for the proofs of the upper bounds in \cref{pr:no-cutoff,th:cutoff}, respectively. 

\subsection{Organization of the paper}
The rest of the paper is organized as follows.  \cref{sec:intert-and-dualities} is devoted to the introduction of several dualities and intertwining relations involving the processes under analysis. In \cref{sec:gap_proof} we prove the spectral gap identity presented in \cref{sec:gap}. The proofs of the mixing results for the Averaging are presented in \cref{sec:avg-results-proofs}, while those for the Binomial Splitting in \cref{sec:proof_cutoff_bin}.

\section{Intertwining and duality relations}\label{sec:intert-and-dualities}
In this section, we present the intertwining and (self-)duality relations involving the Averaging and the Binomial Splitting which will be used all throughout. For a  general account on these two probabilistic tools in the context of Markov processes and interacting particle systems, the interested reader may refer, e.g., to \cite{liggett_interacting_2005-1,giardina_duality_2009,miclo_markovian2018} and references therein.

\begin{proposition}[Multinomial intertwining]\label{pr:mult-inter}
	For all  $k \in\N$ and functions $f:\Omega_k\to \R$, we have
	\begin{equation}\label{eq:intertwining_semigroup}
	\cS^\avg_t	\varLambda_k f  = \varLambda_k \cS_t^{\bin(k)} f\ ,\qquad t\ge 0
	\end{equation}
	where, for all $\eta\in\Delta$,
	\begin{equation}
	\varLambda_k f(\eta)\coloneqq \E_{\mu_{k,\eta}}\left[f\right],
	\end{equation} 
	and $\mu_{k,\eta}$ stands for ${\rm Multinomial}(k,\eta)$. 
\end{proposition}
\begin{proof}
	Recall  \cref{eq:pbin,eq:pavg}. We  prove 
	\begin{equation}\label{eq:intertwining_generator}
	\cP_{xy}^\avg	\varLambda_k f(\eta)  = \varLambda_k \cP_{xy}^{\bin(k)} f(\eta),\qquad \eta\in\Delta,
	\end{equation}
	for all $xy\in E$ and $f:\Omega_k\to \R$; this  yields the analogue of   \cref{eq:intertwining_semigroup} for the corresponding generators, from which \cref{eq:intertwining_semigroup} follows due to the boundedness of the  generators involved. (For notational convenience, in this proof we will write, e.g., $\xi_x$ for $\xi(x)$.) From now on the proof follows by an elementary direct computation.	Starting with the term on the right-hand side in \cref{eq:intertwining_generator}, we rewrite it as
	\begin{align}\label{eq:mult_eq}\nonumber
	\varLambda_k \cP_{xy}^{\bin(k)} f(\eta)&= \sum_{\xi\in \Omega_k}
	\mu_{k,\eta}(\xi)\,\cP^{\bin(k)}_{xy} f(\xi)\\
	&= \sum_{\ell=0}^k \sum_{\xi \in \Omega_k} 	\P\(\left\{{\rm Multinomial}(k,\eta)=\xi\right\}\cap\left\{ \xi_x+\xi_y=\ell\right\}\) \cP_{xy}^{\bin(k)}f(\xi).
	\end{align}
	Notice that, for all $\ell =0,1,\ldots, k$, we have
			\begin{small}
	\begin{align*}
	&\sum_{\xi\in \Omega_k}
	\P\(\left\{{\rm Multinomial}(k,\eta)=\xi\right\}\cap\left\{ \xi_x+\xi_y=\ell\right\}\) \cP_{xy}^{\bin(k)}f(\xi)\\
	&=\sum_{\xi\in \Omega_k} k!  \(\prod_{z\neq x, y}\frac{\eta_z^{\xi_z}}{\xi_z!}\)\frac{\eta_x^{\xi_x}}{\xi_x!}\frac{\eta_y^{\xi_y}}{\xi_y!}\sum_{\xip\in\Omega_k}\ell! \(\ind_{\{\xi'_x+\xi_y'=\ell\}}\prod_{z\neq x, y}\ind_{\{\xi'_z=\xi_z\}}\)\times\\
	&\qquad\times\frac{\left(\frac{\pi_x}{\pi_x+\pi_y} \right)^{\xi'_x}}{\xi'_x!}\frac{\left(\frac{\pi_y}{\pi_x+\pi_y} \right)^{\xi'_y}}{\xi'_y!} f(\xip	)\\
	&=\sum_{\xip\in \Omega_k}k! \left(\prod_{z\neq x, y}\frac{\eta_z^{\xi'_z}}{\xi'_z!}\right)\frac{\left(\frac{\pi_x}{\pi_x+\pi_y} \right)^{\xi_x'}}{\xi_x'!}\frac{\left(\frac{\pi_y}{\pi_x+\pi_y} \right)^{\xi_y'}}{\xi_y'!} \ind_{\{\xi'_x+\xi_y'=\ell\}} \left(\eta_x+\eta_y \right)^\ell f(\xip	) \\
	&=\sum_{\xip\in \Omega_k}k! \left(\prod_{z\neq x, y}\frac{\eta_z^{\xi_z'}}{\xi_z'!}\right)\frac{\left(\frac{\pi_x}{\pi_x+\pi_y}\(\eta_x+\eta_y\) \right)^{\xi_x'}}{\xi_x'!}\frac{\left(\frac{\pi_y}{\pi_x+\pi_y}\(\eta_x+\eta_y\) \right)^{\xi_y'}}{\xi_y'!} \ind_{\{\xi_x'+\xi_y'=\ell\}} f(\xip).
	\end{align*}
	\end{small}
	Plugging this expression back into \cref{eq:mult_eq}, we obtain
	\begin{align*}\nonumber
	\varLambda_k \cP_{xy}^{\bin(k)} f(\eta)
	&= \sum_{\ell=0}^k \sum_{\xi \in \Omega_k} 	\P\(\left\{{\rm Multinomial}(k,\eta)=\xi\right\}\cap\left\{ \xi_x+\xi_y=\ell\right\}\) \cP_{xy}^{\bin(k)}f(\xi)\\
	&= 	\sum_{\xi'\in\Omega_k}\mu_{k,\eta^{xy}}(\xi')\,f(\xi'	),	
	\end{align*}
	which coincides with
	\begin{equation*}
	\cP_{xy}^\avg	\varLambda_k f(\eta)  =\varLambda_k f(\eta^{xy})=\E_{\mu_{k,\eta^{xy}}}[f],
	\end{equation*}
where $\eta^{xy}$ and $\cP_{xy}^\avg$ are defined as in \cref{eq:pavg}. Hence,  the equality in \cref{eq:intertwining_generator} follows.
\end{proof}

Before presenting the  duality relations, it turns out to be convenient to introduce a labeled version of the Binomial Splitting described in \cref{sec:binomial_splitting}. For this reason, for all $k\in \N$, we define the \emph{labeled} Binomial Splitting process with $k$ particles as the  irreducible Markov chain $(\mathbf X_t)_{t\ge 0}$ on $V^k$ whose dynamics is described as follows: start with $k$ labeled particles at positions $\bx=(x_1,\ldots, x_k) \in V^k$; as soon as the Poisson clock of rate $\cxy$ rings,  the particles whose position is either $x$ or $y \in V$, independently of each other, place themselves in $x\in V$ with probability $\frac{\pix}{\pix+\piy}$, in $y \in V$ otherwise. Clearly, this Markov process coincides with the process $(\xi_t)_{t\ge0}$ described in \cref{sec:binomial_splitting} when the labels of the particles are ignored. 
More precisely, the corresponding infinitesimal generator 
\begin{equation}\label{eq:def-gen-labeled}
L^{\bin(k)}= \sum_{xy\in E}\cxy\, L^{\bin(k)}_{xy},
\end{equation}
which is self-adjoint in the $k$-fold tensor space $\cH^{\otimes k}\coloneqq\cH\otimes \cdots\otimes \cH$, where $\cH=L_2(V,\pi)$, maps symmetric functions into symmetric functions, since the particle dynamics does not depend on the particles labels. Rigorously,
\begin{equation}\label{eq:sym}
L^{\bin(k)} \Sym_k = \Sym_k L^{\bin(k)}\ ,
\end{equation} 
where $\Sym_k: \cH^{\otimes k}\to \cH^{\otimes k}$ denotes the orthogonal projector
\begin{equation}\label{eq:sym1}
\Sym_k \psi(x_1,\ldots, x_k)\coloneqq \frac{1}{k!}\sum_{\sigma \in \Sigma_k} \psi(x_{\sigma(1)},\ldots, x_{\sigma(k)})\ ,\qquad \Sym_k^2=\Sym_k\,
\end{equation}
and $\Sigma_k$  the symmetric group on $k$ symbols. We let $(S_t^{\bin(k)})_{t\ge0}$ denote the semigroup associated to the labeled $k$-particle Binomial Splitting. Moreover, for notational convenience, we will adopt the following shorthand:
$$\pi(\bx)=\pi^{\otimes k}(\bx)\coloneqq \pi(x_1)\cdots \pi(x_k), \qquad \bx=(x_1,\dots,x_k)\in V^k.$$

The next two results are not new and variants of them may be found scattered in the literature  in several places for related models, using different techniques, from probabily to Lie algebra. 
For more details about these techniques, we refer the interested reader to, e.g., \cite{aldous_lecture_2012,giardina_duality_2009,redig_generalized_2017,redig_factorized_2018}; below, we provide sketches of their proofs for the reader's convenience.

\begin{proposition}[Self-duality for the Binomial Splitting]\label{proposition:self-duality}
	Let, for all $k,\ell \in \N$, $\bx=(x_1,\ldots, x_k) \in V^k$ and $\xi \in \Omega_\ell$, 
	\begin{equation}
	[\xi]_\bx\coloneqq \xi(x_1)\left(\xi(x_2)-\ind_{x_2=x_1} \right)\cdots \left(\xi(x_k)-\sum_{i=1}^{k-1} \ind_{x_k=x_i} \right)
	\end{equation}
	denote the $\bx$-falling factorial of $\xi$. Then, for all $\bx \in V^k$, $\xi \in \Omega_\ell$ and $t \geq 0$, we have
	\begin{equation}\label{eq:self-dual_rel}
	\E_\xi^{\bin(\ell)}\left[\frac{[\xi_t]_\bx}{\pi(\bx)} \right]= S_t^{\bin(k)} \left(\frac{[\xi]_\cdot}{\pi(\cdot)} \right)(\bx).
	\end{equation}
\end{proposition}
\begin{proof}[Sketch of proof]
	As mentioned above, there are several approaches one might follow to prove this assertion. One option is to proceed by the following two-step argument: first, the self-duality relation between two systems of independent $\bin(1)$ particles, see, e.g. \cite[Proposition 2.9.4]{demasi}; second, recovering the Binomial Splitting processes by ``instantaneous thermalization'', see \cite[Section 6.3]{giardina_duality_2009}.
\end{proof}
\begin{proposition}[Duality between Averaging and Binomial Splitting]\label{pr:dualities}
	Let, for all $\bx \in V^k$ and $\eta \in \Delta$, 
	\begin{equation}\label{eq:duality_function}
	D(\bx,\eta)\coloneqq \prod_{i=1}^k D(x_i,\eta)\coloneqq \prod_{i=1}^k \frac{\eta(x_i)}{\pi(x_i)}
	\end{equation}
	and 
	\begin{equation}\label{eq:duality_function_or}
	\bar D(\bx,\eta)\coloneqq \prod_{i=1}^k \bar D(x_i,\eta)\coloneqq  \prod_{i=1}^k \left( \frac{\eta(x_i)}{\pi(x_i)}-1\right).		
	\end{equation}
	Then, the following duality relations hold: for all $\bx \in V^k$, $\eta \in \Delta$, and $t\ge 0$,
	\begin{equation}\label{eq:dual_rel}
	\E_\eta^\avg\left[D(\bx,\eta_t)\right]= S_t^{\bin(k)}D(\cdot,\eta)(\bx)
	\end{equation}
	and
	\begin{equation}\label{eq:dual_rel_or}
	\E_\eta^\avg\left[\bar D(\bx,\eta_t)\right]= S_t^{\bin(k)} \bar D(\cdot,\eta)(\bx) .
	\end{equation}
\end{proposition}
\begin{remark}[Moment vs.\ orthogonal duality functions]
	Functions of the joint system satisfying  relations as in \cref{eq:self-dual_rel,eq:dual_rel,eq:dual_rel_or}  are usually referred to as duality functions;  more specifically,  functions as in \cref{eq:duality_function} are also known as \textquotedblleft moment\textquotedblright\ or \textquotedblleft classical\textquotedblright\ duality functions, while those as in \cref{eq:duality_function_or} take the name of \textquotedblleft orthogonal\textquotedblright duality functions, due to their relation with orthogonal polynomials (see, e.g., \cite{redig_factorized_2018}).
\end{remark}
\begin{proof}[Sketch of proof]
	Concerning \cref{eq:dual_rel}, the equality follows by poissonizing the Multinomial intertwining in \cref{pr:mult-inter} and acting with this new intertwining on the self-duality functions of \cref{proposition:self-duality} as explained in \cite[Section 5.2]{redig_factorized_2018}. On the other hand, the duality relation in \cref{eq:dual_rel_or} has been proved for the non-thermalized model in, e.g., \cite[Section 5.4]{redig_factorized_2018}; since instantaneous thermalization preserves the duality relations, this concludes the proof.
\end{proof}

\section{Proof of the spectral gap identity}\label{sec:gap_proof}
This section is completely devoted to the proof of \cref{th:gap}. The main ingredient is the following elementary lemma, showing that the duality relation in \cref{eq:dual_rel_or} allows  to map eigenfunctions of the labeled $k$-particle system to candidate eigenfunctions of the Averaging.
\begin{lemma}
	\label{lemma:inner_products_eigenfunctions}
	For $k \geq 1$, let $\psi \in \cH^{\otimes k}$ be an eigenfunction for $-L^{\bin(k)}$ associated to the eigenvalue $\lambda \geq  0$, i.e.,
	\begin{equation}
	L^{\bin(k)} \psi=-\lambda \psi.
	\end{equation}
	Then, $f_\psi \in \cC(\Delta)$ defined as
	\begin{equation}\label{eq:fpsi}
	f_\psi(\eta) \coloneqq \sum_{\bx \in V^k} \pi(\bx)\, \psi(\bx)\,\bar D(\bx,\eta), 
	\end{equation}
	solves 
	\begin{equation}
	\cL^\avg f_\psi=-\lambda f_\psi.
	\end{equation}
\end{lemma}
In other words, \cref{lemma:inner_products_eigenfunctions} shows that, for every eigenfunction $\psi$ of $-L^{\bin(k)}$, either $f_\psi$ is identically zero or $f_\psi$ is an eigenfunction for $-\cL^\avg$ with the same eigenvalue.
\begin{proof}
	By definition, the duality relation \cref{eq:dual_rel_or} and self-adjointness of $L^{\bin(k)}$ in $\cH^{\otimes k}$, we have
	\begin{align}\nonumber
	\cL^\avg f_\psi(\eta)=&\ \sum_{\bx \in V^k}\pi(\bx)\, \psi(\bx)\, \cL^\avg \bar D(\bx,\cdot)(\eta)\\
	=&\ \sum_{\bx \in V^k} \pi(\bx)\, \psi(\bx)\, L^{\bin(k)} \bar D(\cdot,\eta)(\bx)\\
	\nonumber
	=&\ \sum_{\bx \in V^k}\pi(\bx)\, L^{\bin(k)} \psi(\bx)\, \bar 	D(\bx,\eta)=-\lambda f_\psi(\eta).
 \qedhere
	\end{align}
\end{proof}
The forthcoming \cref{proposition:l2-spaces_decompositions}  will determine which eigenfunctions $\psi$ give rise to actual eigenfunctions $f_\psi$. To derive such results we introduce two classes of operators, which are referred to as (particle) \emph{creation} and \emph{annihilation} operators.

We start from the simple observation that the labeled Binomial Spitting is \textquotedblleft consistent\textquotedblright, namely that each    subset  of $k$ labeled $\bin$-particles  still evolves according to the same  Markovian dynamics.
More precisely,  for all $k \in \N$ the generators $L^{\bin(k)}$ and $L^{\bin(k-1)}$ satisfy the following  intertwining relations:
\begin{equation}\label{eq:intertwining}
L^{\bin(k)} \ann_{k,i} = \ann_{k,i} L^{\bin(k-1)} ,
\end{equation}
where, for all $i\in \{1,\dots,k \}$, $\ann_{k,i}: \cH^{\otimes (k-1)}\to \cH^{\otimes k}$ is called the annihilation operator of the $i$th particle for the $k$-particle system; such an operator is one-to-one and defined by
\begin{equation}
(\ann_{k,i})\psi(\bx)\coloneqq \psi(\hat \bx^i),\qquad \hat \bx^i\coloneqq(x_1,\ldots, x_{i-1},x_{i+1},\ldots, x_k).
\end{equation}
The adjoint of $\ann_{k,i}$ is then given by 
$\acr_{k-1,i}: \cH^{\otimes k}\to \cH^{\otimes (k-1)}$, the so-called creation operator of the $i$th particle for the $k$-particle system; such an operator is onto and can be written as
\begin{equation}
(\acr_{k-1,i}) \psi(\hat \bx^i)\coloneqq \sum_{x_i \in V} \pi(x_i)\, \psi(x_1,\ldots,x_i,\ldots,x_k),\qquad \bx\coloneqq(x_1,\ldots, x_k).
\end{equation}
Notice further that $\Img(\ann_{1,1})$ consists of constant functions, while $\Ker(\acr_{0,1})$ consists of  zero mean functions, and that the orthogonal decomposition $\cH= \Ker(\acr_{0,1}) \oplus_{\perp_\pi} \Img(\ann_{1,1})$ holds. The latter decomposition generalizes to the tensor product space $\cH^{\otimes k}$ by 
\begin{equation}\label{eq:decomposition-k}
\cH^{\otimes k}= \Ker(\acr_{k-1}) \oplus_{\perp_\pi} \Img(\ann_k),
\end{equation}
where
\begin{equation}
\Ker(\acr_{k-1})\coloneqq \bigcap_{i=1}^k \Ker(\acr_{k-1,i})=\big(\Ker(\acr_{0,1})\big)^{\otimes k}\quad \text{and}\quad
\Img(\ann_k)\coloneqq \oplus_{i=1}^k \Img(\ann_{k,i}).
\end{equation}
As an immediate consequence of the definition in \cref{eq:dual_rel_or}, it follows that, for all $\eta \in \Delta$, the function $ \bar D(\cdot,\eta):V^k\to\R$ satisfies, for all $1 \leq i \leq k$,
\begin{equation}\label{eq:duality_functions_ker}
(\acr_{k-1,i} \bar D(\cdot,\eta))(\hat \bx_i)= \(\sum_{x_i \in V} \pi(x_i)\, \bar D(x_i,\eta)\) \bar D(\hat \bx_i,\eta)=0\ ,\qquad \bx \in V^k\ ;
\end{equation}
namely, $\bar D(\cdot,\eta)\in \Ker(\acr_{k-1})$.

\begin{lemma}\label{proposition:l2-spaces_decompositions}  Fix $k \in \N$. Given $\psi \in \cH^{\otimes k}$, for $f_\psi\in \cC(\Delta)$ defined as in \cref{eq:fpsi}, we have 
	\begin{equation}
	f_\psi \equiv 0\qquad \iff \qquad \Sym_k\psi \in \Img(\ann_k)\ .
	\end{equation}
\end{lemma}

\begin{proof}
	The implication ``$\Leftarrow$'' is a consequence of the adjointness of $\ann_{k,i}$ and $\acr_{k-1,i}$ with \cref{eq:duality_functions_ker}. Indeed, if $\Sym_k\psi=\ann_{k,i}\phi$ for some $\phi\in \cH^{\otimes k-1}$ and $i\in\{1,\dots, k \}$, 
	\begin{align*}
		f_{\psi}(\eta)=\sum_{\bx \in V^k} \pi(\bx)\, \ann_{k,i}\phi(\bx)\,\bar D(\bx,\eta)=\sum_{\hat\bx_i \in V^{k-1}} \pi(\hat\bx_i)\,\phi(\hat\bx_i)\, \acr_{k,i}\bar D(\hat\bx_i,\eta)=0.
	\end{align*}
	For the ``$\Rightarrow$'' part, thanks to the decomposition in \cref{eq:decomposition-k}, it suffices to show that $f_\psi \neq 0$ for all symmetric and non-zero functions $\psi$ in $ \Ker(\acr_{k-1})$, where $f_\psi \in \cC(\Delta)$ is defined as in \cref{eq:fpsi}.  Thus, fix any such $\psi \in \Ker(\acr_{k-1})\subseteq \cH^{\otimes k}$ and consider the corresponding  $f_\psi \in \cC(\Delta)$. Notice  that, by the definition of $f_\psi$ in \cref{eq:fpsi} and $\bar D(\cdot,\pi)=0$, we have $f_\psi(\pi)=0$; thus, the conclusion follows if we  show  that  $f_\psi(\eta)\neq 0$ for some $\eta \in \Delta\setminus\{\pi\}$.
	
	By the non-degeneracy of $\pi \in \Delta$, there exists $h> 0$ such that
	\begin{equation}\label{eq:tangent_h}
	\pi+\zeta \in \Delta,\qquad \zeta \in \cT_h\coloneqq \left\{\gamma : V\to \R\:\bigg\rvert\:\sum_{x\in V}\gamma(x)=0\ ,\ \left\|\gamma\right\|_\infty\leq h\right\}.	
	\end{equation} 
	Arguing by contradiction, let us suppose that 
	\begin{equation}\label{eq:vanish}
	f_{\psi}(\pi+\zeta)= \sum_{\bx\in V}\pi(\bx)\, \psi(\bx)\, \prod_{i=1}^k \frac{\zeta(x_i)}{\pi(x_i)} = 0
	\end{equation}
	holds for all $\zeta \in \cT_h$. Note that in the first identity we  only used that, by definition of  $\bar D(\bx,\cdot)$, we have $\bar D(\bx,\pi+\zeta)=\prod_{i=1}^k \frac{\zeta(x_i)}{\pi(x_i)}$. Then,  by homogeneity, \cref{eq:vanish} holds for all  $\zeta \in  \cT_\infty$;
	in other words,  $\psi$ is orthogonal to all functions in the linear span of
	\begin{equation}\label{eq:diagonal_tensors}
	\left\{\varphi^{\otimes k}\coloneqq \varphi\otimes\cdots \otimes \varphi\in \Ker(\acr_{k-1})\:\big\rvert\: \varphi  \in \Ker(\acr_{0,1})\right\} .
	\end{equation}
	By the polarization identity, the linear span of \cref{eq:diagonal_tensors} is dense in the  $k$-fold symmetric tensor of $\Ker(\acr_{0,1})$, and since $\psi \in \Ker(\acr_{k-1})=\big(\Ker(\acr_{0,1})\big)^{\otimes k} $ was chosen to be symmetric, $\psi$ must vanish.
	Hence, the assumption that $f_\psi \equiv 0$  yields $\psi\equiv0$, a contradiction.					
\end{proof}

\subsection{Proof of \cref{th:gap}} 
By means of the duality relation in \cref{eq:dual_rel_or}, we constructed in \cref{lemma:inner_products_eigenfunctions,proposition:l2-spaces_decompositions} non-trivial eigenfunctions for $-\cL^\avg$ in terms of suitable symmetric eigenfunctions for $-L^{\bin(k)}$. As we show  in \cref{lemma:lipschitz} below, the main feature of such eigenfunctions for the Averaging process is that, as soon as $k\ge 2$, the point $\pi \in \Delta$ is a zero (recall that $f_\psi(\pi)=0$) of at least the second order  with respect to the $L^2$-distance (in the sense of \cref{eq:416}) on the simplex $\Delta$. This property combined with the $L^2$-Wasserstein contraction result from \cref{pr:aldous-lanoue}---for which we provide the proof below---completes the proof of \cref{th:gap}. 

\begin{proof}[Proof of \cref{pr:aldous-lanoue}]
	The variation of $\|\frac{\eta}{\pi} -1\|_2^2$ after a mass exchange among sites $x, y \in V$ equals	
	\begin{align}\nonumber
	& (\pix+\piy)\left(\frac{\eta(x)+\eta(y)}{\pix+\piy}-1 \right)^2-\pix\left(\frac{\eta(x)}{\pix}-1 \right)^2 -\piy\left(\frac{\eta(y)}{\piy}-1 \right)^2\\
	&=  -\frac{\pi(x)\pi(y)}{\pi(x)+\pi(y)}\left(\frac{\eta(x)}{\pi(x)}-\frac{\eta(y)}{\pi(y)} \right)^2.
	\end{align}
	Therefore, by  definition of spectral gap in \cref{eq:def-spectral-gap} and that of Dirichlet form in \cref{eq:def-dirichlet-bin1}, we obtain
	\begin{equation*}
	\cL^\avg \left\|\frac{\eta}{\pi}-1 \right\|_2^2= - \cE_{\bin(1)}\left(\frac{\eta}{\pi} \right)\leq -\gap_1\left\|\frac{\eta}{\pi}-1 \right\|_2^2.
	\end{equation*}
	An application of  Gr\"onwall inequality yields the desired result.
\end{proof}
\begin{lemma}\label{lemma:lipschitz}
	Consider  $k \geq 2$, $\psi\in \cH^{\otimes k}$ and $f_\psi$ as in \cref{eq:fpsi}. Then,
	\begin{equation}\label{eq:416}
	C_\psi\coloneqq\sup_{\eta \in \Delta\setminus \{\pi\}}\frac{|f_\psi(\eta)|}{\left\|\frac{\eta}{\pi}-1 \right\|_2^2} \in [0,\infty).
	\end{equation}
\end{lemma}
\begin{proof}
	By Cauchy-Schwarz inequality, we obtain
	\begin{equation}
	|f_\psi(\eta)|=\left|\sum_{\bx \in V^k}\pi(\bx)\, \psi(\bx)\, \bar D(\bx,\eta) \right|\leq \sqrt{\sum_{\bx \in V^k} \pi(\bx)\left(\psi(\bx) \right)^2} \sqrt{\sum_{\bx \in V^k}\pi(\bx)\left(\bar D(\bx,\eta) \right)^2}\ .
	\end{equation}
	Because of the product structure of both probability measures $\pi=\pi^{\otimes k}: V^k\to \R$ and duality functions $\bar D(\cdot,\eta)$, we have
	\begin{equation}
	\sqrt{\sum_{\bx \in V^k}\pi(\bx)\left(\bar D(\bx,\eta) \right)^2}= \left(\sum_{x \in V}\pi(x)\left(\bar D(x,\eta) \right)^2 \right)^{\frac{k}{2}}= \left\|\frac{\eta}{\pi}-1 \right\|_2^k\ .
	\end{equation}
	The desired conclusion follows because $\Delta$ is compact and the $L^2$-norm is continuous.
\end{proof}
In view of the above lemmas, we may conclude by employing a well-known argument due to  Chen and  Wang (\cite{chen1997estimation}):
\begin{proof}[Proof of \cref{th:gap}]
	Clearly, we need to consider only the case $k \geq 2$. The inequality
	\begin{equation}
	\gap_k\leq \gap_1
	\end{equation}
	follows at once from  \cref{eq:sym,eq:intertwining}.  Indeed,  calling $\xi^x\in \Omega_{k-1}$  the configuration obtained from $\xi\in\Omega_k$ by removing one particle at $x\in V$, the linear operator  $J_k:L^2(\Omega_{k-1},\mu_{k-1,\pi})\to L^2(\Omega_k, \mu_{k,\pi})$ defined as
	\begin{equation}\label{eq:def-Jk}
	J_kf(\xi)\coloneqq\sum_{x\in V}\xi(x) f(\xi^x),\qquad \xi\in \Omega_k,
	\end{equation}
	is injective; moreover, since its action corresponds to that of a symmetrized annihilation operator on symmetric functions, $J_k$ satisfies
	\begin{equation}\label{eq:intertwining-Jk}
	\cL^{\bin(k)}J_k=J_{k}\cL^{\bin(k-1)},\qquad k\ge 2,
	\end{equation}
	see \cref{eq:intertwining}.
	
	As for the reverse inequality, due to self-adjointness of $-L^{\bin(k)}$ and the decomposition of $\cH^{\otimes k}$ in \cref{eq:decomposition-k}, it suffices to consider symmetric eigenfunctions in $\Ker(\acr_{k-1})$. Let  $\psi$ be such an  eigenfunction for $-L^{\bin(k)}$ with corresponding eigenvalue $\lambda > 0$ and $f_\psi\in \cC(\Delta)$  defined  as in \cref{eq:fpsi}. As already noted in \cref{proposition:l2-spaces_decompositions}, $f_\psi(\pi)=0$ and $f_\psi \neq 0$; moreover,  by \cref{lemma:inner_products_eigenfunctions}, $\cL^\avg f_\psi = -\lambda f_\psi$.  Then, for all  $t \geq 0$ and $\eta \in \Delta \setminus\{\pi\}$ such that $f_\psi(\eta)\neq 0$, we have
	\begin{equation}\label{eq:trick}
	\begin{split}
	e^{-\lambda t}|f_\psi(\eta)|= |\E^\avg_\eta[f_\psi(\eta_t)]|
	&\leq\E^\avg_\eta\left[|f_\psi(\eta_t)|\right]\\
	&\leq C_\psi\, \E^\avg_\eta\left[\left\|\frac{\eta_t}{\pi}-1 \right\|_2^2\right]
	\leq C_\psi\, e^{-\gap_1 t}\left\|\frac{\eta}{\pi}-1 \right\|_2^2,
	\end{split}
	\end{equation}
	where $C_\psi > 0$ is the constant introduced in  \cref{lemma:lipschitz}, whereas   the last step follows from  \cref{pr:aldous-lanoue}. Since, again by  \cref{lemma:lipschitz}, $\frac{|f_\psi(\eta)|}{C_\psi \left\|\frac{\eta}{\pi}-1\right\|_2^2} \in (0,1]$, we further obtain 
	\begin{equation}
	\lambda\geq \gap_1+\frac{\log\left(\frac{|f_\psi(\eta)|}{C_\psi \left\|\frac{\eta}{\pi}-1 \right\|_2^2} \right)}{t}
	\end{equation}
	for all $t > 0$. Taking $t\to\infty$ yields the desired result.
\end{proof}

\section{Proofs from \cref{sec:avg-results}}\label{sec:avg-results-proofs}
In this section we prove the mixing results for the Averaging stated in \cref{sec:avg-results}. 
The section is divided in four parts. First in \cref{suse:lb-avg} we prove an easy lower bound for the $L^p$-Wasserstein distance to equilibrium. We then extract two upper bounds which will be used to estimate the distance from equilibrium at different scales. On the one hand, in \cref{suse:ub-short-avg}, we show that a time $\Theta(t_\rel)$ suffices to bring the mean $L^2$-distance arbitrarily close to zero. On the other hand, in \cref{suse:ub-long-avg}, we use this latter bound to control the mixing at times $(1+o(1))\frac{t_\rel}{2} \log(k)$, showing that this is sufficient to shrink the  $L^2$-Wasserstein distance down further to $o(k^{-1/2})$.  For this reason, we are going to refer to times of  order $\Theta(t_\rel)$ as ``short times" and to times of order $\Theta(t_\rel\log(k))$ as ``longer times''.

As mentioned in \cref{sec:avg-results}, the latter is shown to hold as long as $k=O(n^2)$. Finally, in \cref{suse:proofs-avg} we collect all results of these subsections to prove \cref{pr:no-cutoff-avg,th:mixing-avg}.
\subsection{Lower bound}\label{suse:lb-avg}
The next lemma follows easily by using the duality relations between the Averaging and the one-particle system in \cref{pr:dualities}. In what follows, in analogy with  \cref{eq:def-hx}, we define
\begin{equation}\label{eq:def-h-eta-t}
	h^{\eta}_t(x)\coloneqq S_t^{\bin(1)} \(\frac{\eta}{\pi} \)(x)= S_t^{\bin(1)} D(\cdot,\eta)(x),\qquad x\in V,\:\eta\in\Delta,\:t\ge 0.
\end{equation}
\begin{lemma}
	\label{lemma:lower_bound_general}
	For all $\eta\in \Delta$, $p \in [1,\infty]$ and $t \geq 0$, we have
		\begin{align}\label{eq:pr-lb-w200}
	\E^\avg_\eta\left[\left\|\frac{\eta_t}{\pi}-1 \right\|_p \right]\geq \|h_t^\eta-1\|_p.
	\end{align}
	As a consequence,
	\begin{align}\label{eq:pr-lb-w2}
	\sup_{\eta \in \Delta}\E^\avg_\eta\left[\left\|\frac{\eta_t}{\pi}-1 \right\|_p \right]\geq\sup_{\eta\in\Delta} \|h_t^\eta-1\|_p\ge e^{- \frac{t}{t_\rel}}.
	\end{align}
\end{lemma}
\begin{proof}
	Let $q\coloneqq q(p)\in [1,\infty]$ denote the conjugate exponent of $p \in [1,\infty]$. Then, by the dual formulation  of $\left\|\cdot\right\|_p$ and \cref{pr:dualities} with $k=1$, we have
	\begin{align*}
	\E^\avg_\eta\left[\left\|\frac{\eta_t}{\pi}-1 \right\|_p \right]=&\ \E^\avg_\eta\left[ \sup_{\left\|\psi\right\|_q= 1}\sum_{x\in V}\pi(x)\, \bar D(x,\eta_t)\,\psi(x) \right]\\ 
	\geq&\ \sup_{\left\|\psi\right\|_q= 1}\sum_{x \in V} \pi(x)\, S_t^{\bin(1)}\bar D(\cdot,\eta)(x)\, \psi(x)  
	= \|h_t^\eta-1\|_p\: .
	\end{align*}
This shows \cref{eq:pr-lb-w200}.
	Then  \cref{eq:pr-lb-w2} follows by passing to the supremum in $\eta\in\Delta$, and using the monotonicity of $L^p$-norms and \cite[Lemma 20.11]{levin2017markov} with $p=1$.
\end{proof}
\subsection{Upper bound for short times}\label{suse:ub-short-avg}
Let us start by noting that, by the duality relation between the Averaging and the two-particle system we have, for all $t\ge0$ and $\eta\in\Delta$,
\begin{align}
\label{eq:435}	\E^{\avg}_\eta\[\left\|\frac{\eta_t}{\pi}-1 \right\|_2^2 \]&=\sum_{x\in V}\pix\: \E^{\avg}_\eta\[\(\frac{\eta_t(x)}{\pi(x)} \)^2-1 \]\\
\label{eq:436}	&=\sum_{x\in V}\pi(x) \( S^{\bin(2)}_tD(\cdot,\eta)(x,x) -1 \)\\
\label{eq:437}	&=\sum_{x\in V}\pi(x) \(\sum_{y,z\in V} p^{\bin(2)}_t((y,z),(x,x))\frac{\eta(y)\eta(z)}{\pi(x)^2} -1 \)\\
\label{eq:438}	&=\sum_{y,z\in V}\eta(y)\eta(z)\sum_{x\in V}\pix\(\frac{p^{\bin(2)}_t((y,z),(x,x))}{\pix^2}-1\)\\
\label{eq:1-inf-bin2}&\leq\ \max_{x,y,z,w\in V}\left|\frac{p^{\bin(2)}_t((x,y),(z,w))}{\pi(z)\pi(w)}-1 \right|\ ,
\end{align}
where we used the symbol
\begin{equation}\label{eq:def-transitions}
p_t^{\bin(k)}(\bx,\by)\coloneqq \P^{\bin(k)}\( \mathbf{X}_t=\by \:|\: \mathbf{X}_0=\bx \),\qquad k\in\N,\:\bx,\by\in V^k,\: t\ge 0\ ,
\end{equation}
to refer to the transition probabilities of the labeled particle system. We remark that \cref{eq:436} follows by \cref{eq:dual_rel}, \cref{eq:437} is just reversibility and to obtain \cref{eq:1-inf-bin2} we used $\eta,\pi\in\Delta$.

In order to control the quantity in \cref{eq:1-inf-bin2} we now derive a Nash inequality for the two-particle system from the analogous one for $\bin(1)$.  The main ingredient is the following comparison result between the two-particle system, i.e., $\bin(2)$, and the product chain of two one-particle systems, i.e., $\bin(1)\otimes\bin(1)$.

\begin{lemma}[Comparison of Dirichlet forms]\label{lemma:dirichlet_forms_comparison}
	For all $\psi\in \cH^{\otimes 2}$,	
	\begin{equation}\label{eq:dirichlet_form_ordering}
	\frac{1}{2}\, \cE_{\bin(1)\otimes \bin(1)}(\psi)\leq \cE_{\bin(2)}(\psi)\leq \cE_{\bin(1)\otimes\bin(1)}(\psi)\ ,
	\end{equation}
	where $\cE_{\bin(2)}$ and $\cE_{\bin(1)\otimes \bin(1)}$ denote the Dirichlet forms on $\cH^{\otimes2}$ of the corresponding processes.
	
\end{lemma}
\begin{proof}
	Recall that 
	\begin{equation}
	L^{\bin(2)}= \sum_{xy \in E}\cxy\, L^{\bin(2)}_{xy}\quad \text{and}\quad  L^{\bin(1)\otimes\bin(1)}=\sum_{xy\in E} \cxy\left(L^{\bin(1)}_{xy}\oplus L^{\bin(1)}_{xy} \right) ,
	\end{equation}
with $A\oplus B:= A\otimes \identity + \identity\otimes B$ denoting the Kronecker sum of two operators.
	For all $xy\in E$ and $\psi \in  \cH^{\otimes 2}$, it is simple to check that
			\begin{small}
	\begin{multline}
	\left(L^{\bin(2)}_{xy}-L^{\bin(1)}_{xy}\oplus L^{\bin(1)}_{xy}\right)\psi(z,w)=\\
	 \begin{cases}
	\left(\frac{\pi(y)}{\pi(x)+\pi(y)}\right)^2 \left(\psi(y,y)+\psi(x,x)-\psi(x,y)-\psi(y,x) \right)
	&\text{if}\ z=w=x\\
	\left(\frac{\pi(x)}{\pi(x)+\pi(y)}\right)^2 \left(\psi(y,y)+\psi(x,x)-\psi(x,y)-\psi(y,x) \right) &\text{if}\ z=w = y\\
	-\left(\frac{\pi(x)}{\pi(x)+\pi(y)}\right)\left(\frac{\pi(y)}{\pi(x)+\pi(y)}\right)\left(\psi(y,y)+\psi(x,x)-\psi(x,y)-\psi(y,x) \right)
	&\text{if}\ (z,w)=(x,y)\text{ or } (y,x)\\
	0&\text{otherwise}\ .
	\end{cases}
	\end{multline}
\end{small}
	Combining the above two identities, we obtain,   for all $\phi, \varphi\in \cH^{\otimes 2}$,
	\begin{align}\label{eq:strange_object}
	& 	\sum_{z,w \in V} \pi(z)\pi(w)\, \phi(z,w)\left(L^{\bin(2)} - L^{\bin(1)\otimes \bin(1)} \right)\varphi(z,w)\\
	\nonumber=&\ \sum_{xy\in E} \cxy \left(\frac{\pi(x)\pi(y)}{\pi(x)+\pi(y)} \right)^2 \left( \varphi(x,x) + \varphi(y,y) -\varphi(x,y)-\varphi(y,x)\right)\times\\\nonumber&\qquad\times \left( \phi(x,x) + \phi(y,y) -\phi(x,y)-\phi(y,x)\right).
	\end{align}
	Therefore, for all $\psi\in \cH^{\otimes 2}$, 
	\begin{align}\nonumber
	\label{eq:strange_object2}
	\cF_{\bin(2)}(\psi)\coloneqq&\ 	\sum_{z,w \in V} \pi(z)\pi(w)\, \psi(z,w)\left(L^{\bin(2)} - L^{\bin(1)\otimes\bin(1)} \right)\psi(z,w)\\
	=&\ \sum_{xy\in E} \cxy \left(\frac{\pi(x)\pi(y)}{\pi(x)+\pi(y)} \right)^2 \left( \psi(x,x) + \psi(y,y) -\psi(x,y)-\psi(y,x)\right)^2 \geq 0 \ .
	\end{align}
	As a consequence of the definition of $\cF_{\bin(2)}$, we get	
	\begin{align}\label{eq:partial_dirichlet_form_3}
	\cE_{\bin(2)}(\psi) &= \cE_{\bin(1)\otimes \bin(1)}(\psi) -\cF_{\bin(2)}(\psi)
	\ ,
	\end{align}
	yielding, since $\cF_{\bin(2)}(\psi)\geq 0$, the second inequality in \cref{eq:dirichlet_form_ordering}. For what concerns the first inequality in \cref{eq:dirichlet_form_ordering},  we show that
	\begin{equation}\label{eq:half_dirichlet_form}
	\frac{1}{2}\, \cE_{\bin(1)\otimes \bin(1)}(\psi)-\cF_{\bin(2)}(\psi)\geq 0
	\end{equation}
	holds for all $\psi \in \cH^{\otimes 2}$.  For this purpose, let us  recall that
	\begin{align*}
	&\cE_{\bin(1)\otimes \bin(1)}(\psi) \\&=  \sum_{xy\in E}\cxy\, \frac{\pi(x)\pi(y)}{\pi(x)+\pi(y)}
	\left\{\sum_{z \in V}\pi(z)\left(\left(\psi(x,z)-\psi(y,z) \right)^2 + \left(\psi(z,x)-\psi(z,y)\right)^2	 \right) \right\},
	\end{align*}
	therefore, \cref{eq:half_dirichlet_form} is equivalent to show that
	\begin{small}
	\begin{align}\nonumber
	&\ \sum_{xy\in E}\cxy\, \frac{\pi(x)\pi(y)}{\pi(x)+\pi(y)}\\
	\nonumber
	\times&\   \left\{\sum_{z \in V}\pi(z)\left(\psi(x,z)-\psi(y,z) \right)^2 - \frac{\pi(x)\pi(y)}{\pi(x)+\pi(y)} \left(\left(\psi(x,x)-\psi(y,x) \right)+\left(\psi(y,y)-\psi(x,y) \right) \right)^2 \right\}\\
	\nonumber
	&+ \sum_{xy\in E}\cxy\, \frac{\pi(x)\pi(y)}{\pi(x)+\pi(y)}\\
	\nonumber	\times&   \left\{\sum_{z \in V}\pi(z)\left(\psi(z,x)-\psi(z,y) \right)^2 - \frac{\pi(x)\pi(y)}{\pi(x)+\pi(y)} \left(\left(\psi(x,x)-\psi(x,y) \right)+\left(\psi(y,y)-\psi(y,x) \right) \right)^2 \right\}
	\end{align}
\end{small}is non-negative. Now, we claim that this  holds because each expression between curly brackets is non-negative.  Indeed, for all $xy \in E$, focusing on the first expression between curly brackets (the second one can be dealt with analogously) and setting 
	\begin{align}
	\begin{split}
	u&\coloneqq \left(\psi(x,x)-\psi(y,x) \right),\\
	v&\coloneqq \left(\psi(y,y)-\psi(x,y) \right),
	\end{split}
	\begin{split}
	p&\coloneqq \pi(x),\\
	q&\coloneqq \pi(y)\ ,
	\end{split}	
	\end{align}
	we have
	\begin{align*}
	\sum_{z\neq x,y}& \pi(z)\left(\psi(x,z)-\psi(y,z) \right)^2 + \left( pu^2 + qv^2 -  \frac{pq}{p+q}\left(u+v \right)^2\right) \\
	& \geq pu^2 + qv^2 -  \frac{pq}{p+q}\left(u+v \right)^2 \geq 0\ .
	\end{align*}
	This concludes the proof.
\end{proof}

\begin{remark}\label{rem:gap-labeled}
	The first inequality in \cref{lemma:dirichlet_forms_comparison} and the intertwining relation \cref{eq:intertwining} ensure that not only $\gap_2=\gap$, as shown in \cref{th:gap}, but also that the spectral gap of the \emph{labeled} two-particle system equals $\gap$.  
\end{remark}

\begin{proposition}\label{proposition:nash}
	Under \cref{assumption:nash}, there exist $C,c>0$, independent of $n$, such that 
	\begin{equation}\label{eq:qqq}
	\max_{x,y,z,w\in V}\left|\frac{p_t^{\bin(2)}((x,y),(z,w))}{\pi(z)\pi(w)}-1 \right|\leq c\, e^{-\frac{t}{t_\rel}}, \qquad t\ge C t_\rel.
	\end{equation}
	Therefore, as a consequence of the inequalities in \cref{eq:435,eq:436,eq:437,eq:438,eq:1-inf-bin2}, we have
	\begin{equation}\label{eq:518}
	\sup_{\eta\in\Delta}\E^{\avg}_\eta\[\left\|\frac{\eta_t}{\pi}-1 \right\|_2^2 \]\le c\, e^{-\frac{t}{t_\rel}},\qquad t\ge Ct_\rel.
	\end{equation} 
\end{proposition}
\begin{proof}
	By \cite[Theorem\ 2.3.4]{saloff1997lectures}, \cref{assumption:nash} implies
	\begin{equation}\label{eq:assumption_nash2_bin1}
	\max_{x,y\in V}\frac{p_t^{\bin(1)}(x,y)}{\pi(y)} \leq e \left(\frac{d \nash}{2t} \right)^{\frac{d}{2}},\qquad t\le T,
	\end{equation}
	and an analogous inequality for the product chain $\bin(1)\otimes\bin(1)$:  letting $\dtwo\coloneqq 2 d$,
	\begin{equation}
	\max_{x,y,z,w\in V}\frac{p_t^{\bin(1)\otimes\bin(1)}((x,y),(z,w))}{\pi(z)\pi(w)}\leq e^2\left(\frac{d'\nash}{t} \right)^{\frac{\dtwo}{2}}\ ,\qquad t \leq T\ ,
	\end{equation}
	where $p_t^{\bin(1)\otimes\bin(1)}$ is defined in analogy to \cref{eq:def-transitions}. Because of reversibility of $\bin(1)\otimes\bin(1)$, 
	the converse to Nash's argument due to \cite{carlen_upper_1986} (see also \cite[Theorem\ 2.3.7]{saloff1997lectures}) ensures that
	\begin{equation}\label{eq:this}
	\left\|\psi \right\|_2^{2\left(1+\frac{2}{\dtwo} \right)}\leq 	C'\nash\left( \cE_{\bin(1)\otimes\bin(1)}(\psi)+\frac{1}{2T}\left\|\psi \right\|_2^2 \right)\left\|\psi \right\|_1^{\frac{4}{\dtwo}},\qquad \forall \psi \in \cH^{\otimes 2},
	\end{equation}
	for some constant $C'>0$ depending only on $d$.
	By combining \cref{eq:this} with the first  inequality  in  \cref{lemma:dirichlet_forms_comparison}, we further obtain the following integral version of Nash inequality  for $\bin(2)$:
	\begin{equation}\label{eq:this2}
	\left\|\psi \right\|_2^{2\left(1+\frac{2}{\dtwo} \right)}\leq 	2C'\nash\left( \cE_{\bin(2)}(\psi)+\frac{1}{4T}\left\|\psi \right\|_2^2 \right)\left\|\psi \right\|_1^{\frac{4}{\dtwo}},\qquad \forall \psi \in \cH^{\otimes 2},
	\end{equation}
	which, again by  \cite[Theorem\ 2.3.4]{saloff1997lectures}, implies
	\begin{equation}\label{eq:nash_bin_two}
	\max_{x,y,z,w\in V}\frac{p^{\bin(2)}_t((x,y),(z,w))}{\pi(z)\pi(w)}  \leq e\left(\frac{\dtwo C' \nash}{t} \right)^{\frac{\dtwo}{2}}\ ,\qquad t \leq 4T .
	\end{equation}
	
	By  Chapman-Kolmogorov equation, as well as Cauchy-Schwarz and Poincar\'e inequalities, we obtain, for all $t\geq s\geq 0$, 
	\begin{align}\nonumber
	\max_{x,y,z,w\in V}\left|\frac{p_t^{\bin(2)}((x,y),(z,w))}{\pi(z)\pi(w)}-1 \right|\leq&\ \max_{x,y\in V}\left\|\frac{p_{t/2}^{\bin(2)}((x,y),\cdot)}{\pi\otimes\pi}-1\right\|_{\cH^{\otimes 2}}^2\\
	\leq&\ \exp\left\{- \frac{t-s}{t_\rel}   \right\}  \max_{x,y\in V} \left\|\frac{p_{s/2}^{\bin(2)}((x,y),\cdot)}{\pi\otimes\pi}-1 \right\|_{\cH^{\otimes 2}}^2\ ,
	\end{align} 
	where the last inequality follows by \cref{rem:gap-labeled}. Setting $s\coloneqq 8T=\Theta(\nash)$ in the last displacement,  \cref{eq:nash_bin_two} yields, for all $t \geq 8T \geq 0$ and some constant $c>0$,
	\begin{equation}\label{eq:qqq}
	\max_{x,y,z,w\in V}\left|\frac{p_t^{\bin(2)}((x,y),(z,w))}{\pi(z)\pi(w)}-1 \right|\leq c\exp\left\{-\frac{t}{t_\rel}\right\},
	\end{equation}
	where we used the fact that $T=O(t_\rel)$ thanks to \cref{assumption:nash}.
\end{proof}
\subsection{Upper bound for longer times}\label{suse:ub-long-avg}
The next proposition provides the upper bound which will turn out to be central for the proof of \cref{th:mixing-avg}.
\begin{proposition}\label{lemma:upper-bound-avg-large}
	Under \cref{assumption:nash,assumption:unif_ellipticity}, there exists a  constant $C'>0$ such that for all $t>C t_\rel$, with $C$ as in \cref{proposition:nash}, and $\eta\in\Delta$,
	\begin{equation}\label{eq:lemma-long-times}
	\E_{\eta}^\avg\left[\left\|\frac{\eta_t}{\pi}-1 \right\|^2_2 \right]\le C'\: \( \frac{ 1}{n} \exp\left\{-\frac{t}{t_\rel} \right\} + \exp\left\{-\frac{2\, t}{t_\rel} \right\}  \) \left\|\frac{\eta}{\pi}-1 \right\|_2^2\ .
	\end{equation}
\end{proposition}
\begin{proof}
	We start by rewriting the left-hand side of \cref{eq:lemma-long-times} as
	\begin{align}\label{eq:l2norm_eta}\nonumber
	&\E^\avg_{\eta}\left[\left\|\frac{\eta_t}{\pi}-1 \right\|^2_2 \right]=  \sum_{x \in V} \pi(x)\, 	\E^\avg_{\eta}\left[\left(\frac{\eta_t(x)}{\pi(x)}\right)^2-1  \right]\\
	\nonumber
	&=\sum_{x \in V}\pi(x)\left(\left( \E^\avg_{\eta}\left[\frac{\eta_t(x)}{\pi(x)} \right]\right)^2 -1\right) \\
	&\nonumber\:\:+  \sum_{x \in V}\pi(x)\, \left(\E^\avg_{\eta}\left[\left(\frac{\eta_t(x)}{\pi(x)}\right)^2\right]-\left(\E^\avg_{\eta}\left[\frac{\eta_t(x)}{\pi(x)} \right]\right)^2\right)
	\\&= \left\|h^{\eta}_t-1 \right\|_2^2+  \sum_{z \in V} \pi(z)	\left(S_t^{\bin(2)}- S_t^{\bin(1)\otimes \bin(1)} \right)
	\(\frac{\eta}{\pi}\otimes\frac{\eta}{\pi} \)(z,z),
	\end{align}
	where we used the duality in \cref{eq:dual_rel} with $k=1$ and $k=2$ and the definition in \cref{eq:def-h-eta-t}. The rest of the proof is devoted to estimating \cref{eq:l2norm_eta}: while the bound for the first term is straightforward (see \cref{eq:first-term}), as for the second term, first we rewrite it in \cref{eq:Ft} below, then we split it into two parts in \cref{eq:second-integral}, and conclude  bounding these two expressions in \cref{eq:second-integral,eq:first-integral2}.
	 
	The first term in \cref{eq:l2norm_eta}, by Poincar\'e inequality, can be bounded from above by
	\begin{equation}\label{eq:first-term}
	\left\|h^{\eta}_t-1 \right\|_2^2\le \exp\left\{-\frac{2\: t}{t_\rel}\right\} \left\|\frac{\eta}{\pi}-1 \right\|_2^2.
	\end{equation}
	As for the second term in \cref{eq:l2norm_eta}, noting that 
	$S_t^{\bin(1)\otimes \bin(1)}\left(\frac{\eta}{\pi}\otimes \frac{\eta}{\pi} \right)= h^{\eta}_t\otimes h^{\eta}_t,$
	by the integration by parts formula (see, e.g., \cite[Proposition\ VIII.1.7]{liggett_interacting_2005-1}) we obtain
	\begin{align}\nonumber
	&\cN_t(\eta)\coloneqq\  \sum_{z \in V} \pi(z)	\left(S_t^{\bin(2)}- S_t^{\bin(1)\otimes \bin(1)} \right)
	\left(\frac{\eta}{\pi}\otimes \frac{\eta}{\pi}  \right)(z,z)\\
	\nonumber
	=&\ \sum_{z \in V} \pi(z) \int_0^t S_{t-s}^{\bin(2)}\left(L^{\bin(2)}-L^{\bin(1)\otimes \bin(1)} \right)S^{\bin(1)\otimes\bin(1)}_s \left(\frac{\eta}{\pi}\otimes \frac{\eta}{\pi}  \right)(z,z)\, \dd s\\
	\nonumber
	=&\  \sum_{z \in V}\pi(z)\int_0^t \sum_{v,w \in V} p^{\bin(2)}_{t-s}((z,z),(v,w))\left(L^{\bin(2)}-L^{\bin(1)\otimes \bin(1)} \right)\left(h^{\eta}_s\otimes h^{\eta}_s\right)(v,w)\, \dd s.
	\end{align}
Hence, by the explicit computation in \cref{eq:strange_object} and the duality relation in \cref{eq:dual_rel} with $k=1,2$ we have
\begin{align}\label{eq:Ft}
	\cN_t(\eta)	=&\ \int_0^t  \sum_{xy\in E} \cxy \frac{\pi(x)\pi(y)}{\pi(x)+\pi(y)}\left(h^{\eta}_s(x)-h^{\eta}_s(y) \right)^2 \Phi_{t-s}(x,y)\, \dd s\ ,
	\end{align}
	where
	\begin{align}\label{eq:Phi}
	&\Phi_{t-s}(x,y)\coloneqq \frac{\pi(x)\pi(y)}{\pi(x)+\pi(y)} \sum_{z \in V} \pi(z)\,\E^\avg_{\delta_z}\left[\left(D(x,\eta_{t-s})-D(y,\eta_{t-s}) \right)^2 \right]\\
	\nonumber&\qquad= \frac{\pi(x)\pi(y)}{\pi(x)+\pi(y)} \sum_{z \in V} \pi(z)\left(g_{t-s}^{x,x}(z,z)+g_{t-s}^{y,y}(z,z)- g_{t-s}^{x,y}(z,z)-g_{t-s}^{y,x}(z,z) \right)\ ,
	\end{align}
	with
	\begin{equation}
	g_{t-s}^{x,y}(z,z)\coloneqq \frac{p_{t-s}^{\bin(2)}\left((x,y),(z,z)\right)}{\pi(z)\pi(z)}=S_{t-s}^{\bin(2)}\left( \frac{\ind_{(z,z)}(\cdot,\cdot)}{\pi\otimes \pi}\right)(x,y)\ .
	\end{equation}
	Concerning the function $t\mapsto\Phi_t(x,y)$, it is easy to check that
	\begin{equation}
	\Phi_0(x,y)=1\ ,\qquad x, y \in V\ \text{such that}\ xy\in E\ .
	\end{equation}
	Further,  \cref{assumption:unif_ellipticity} yields
	\begin{equation}\label{eq:Phi_uniform_boundedness}
	\left|	\Phi_{t-s}(x,y)\right|\leq 2\, c_{\rm ell }\ ,\qquad t-s\geq 0\ .
	\end{equation}
Indeed, since $\sum_{z\in V}p^{\bin(2)}_{t-s}((a,b),(z,z))\in [0,1]$ for $a,b \in V$,
\begin{align*}
&	|\Phi_{t-s}(x,y)|\le \frac{\pi(x)\pi(y)}{\pi(x)+\pi(y)}\\
&\times  2\max\left\{\sum_{z\in V}\frac{p^{\bin(2)}_{t-s}((x,x),(z,z))}{\pi(z)},\sum_{z\in V}\frac{p^{\bin(2)}_{t-s}((y,y),(z,z))}{\pi(z)},\sum_{z\in V}\frac{p^{\bin(2)}_{t-s}((x,y),(z,z))}{\pi(z)}\right\},
\end{align*}
and \cref{eq:Phi_uniform_boundedness}  follows at once 	 estimating the above maximum by $(\min_{z\in V}\pi(z))^{-1}$,   $\frac{\pi(x)}{\pi(x)+\pi(y)}\le 1$, and \cref{assumption:unif_ellipticity}. 

	We now split, for some $r\in(0,t)$ to be fixed later, the integral in \cref{eq:Ft} as
	\begin{equation}\label{eq:Nt}
	\begin{split}
	\cN_t(\eta)&=\int_0^r \sum_{xy\in E} \cxy \frac{\pi(x)\pi(y)}{\pi(x)+\pi(y)}\left(h^{\eta}_s(x)-h^{\eta}_s(y) \right)^2 \Phi_{t-s}(x,y) \,  \dd s \\
	&\:\:+ \int_r^t \sum_{xy\in E} \cxy \frac{\pi(x)\pi(y)}{\pi(x)+\pi(y)}\left(h^{\eta}_s(x)-h^{\eta}_s(y) \right)^2 \Phi_{t-s}(x,y)\, \dd s.
	\end{split}
	\end{equation}
	The second term on the right-hand side in \cref{eq:Nt}, thanks to \cref{eq:Phi_uniform_boundedness}, is bounded by 
	\begin{equation}
	\begin{split}
	\int_r^t \sum_{xy\in E} \cxy \frac{\pi(x)\pi(y)}{\pi(x)+\pi(y)}\left(h^{\eta}_s(x)-h^{\eta}_s(y) \right)^2 \Phi_{t-s}(x,y) \,\dd s\\ \le c_{\rm ell} \int_r^t 2 \cE_{\bin(1)}(h_s^\eta) \,\dd s \le c_{\rm ell}  \left\|h^{\eta}_r-1 \right\|_2^2 .
	\end{split}
	\end{equation}
	Hence, by an application of Poincar\'e inequality we obtain
	\begin{equation}\label{eq:second-integral}
	\begin{split}
	\int_r^t \sum_{xy\in E} \cxy \frac{\pi(x)\pi(y)}{\pi(x)+\pi(y)}\left(h^{\eta}_s(x)-h^{\eta}_s(y) \right)^2 \Phi_{t-s}(x,y)\, \dd s\\\le c_{\rm ell} \exp\left\{-\frac{2\:r}{t_\rel} \right\}\: \left\|\frac{\eta}{\pi}-1 \right\|_2^2.
	\end{split}
	\end{equation}
	Recalling that $t >C t_\rel$, we can choose $r=t-C t_\rel$ so that, by \cref{proposition:nash}, for all $s\in[0,r]$,	
	\begin{equation}\label{eq:Phi_nash}
	\left|\Phi_{t-s}(x,y)\right|\leq\frac{ 4 \:c_{\rm ell}\: c}{n}\: \exp\left\{-\frac{t-s}{t_\rel}\right\},
	\end{equation}
	where the constant $c$ is the same as in \cref{proposition:nash}.
	Therefore, the first term on the right-hand side in \cref{eq:Nt} can be bounded by
	\begin{equation}\label{eq:first-integral}
	\begin{split}
	\int_0^r \sum_{xy\in E} \cxy \frac{\pi(x)\pi(y)}{\pi(x)+\pi(y)}\left(h^{\eta}_s(x)-h^{\eta}_s(y) \right)^2 \Phi_{t-s}(x,y)\, \dd s\\\le\frac{ 2 \:c_{\rm ell}\: c}{n} \int_0^r 2\cE_{\bin(1)}(h^\eta_s)e^{-\frac{t-s}{t_\rel}} \, \dd s.
	\end{split}
	\end{equation}
	By an integration by parts and  Poincar\'e inequality, we obtain
	\begin{align*}
	&\int_0^r 2\cE_{\bin(1)}(h^\eta_s)e^{-\frac{t-s}{t_\rel}}  \dd s\\
	&\qquad=e^{-\frac{t}{t_\rel}}\:\left\| \frac{\eta}{\pi}-1\right\|_2^2-e^{-\frac{t-r}{t_\rel}}\:\| h_r^\eta-1\|_2^2+\frac{1}{t_\rel} \int_0^r e^{-\frac{t-s}{t_\rel}}\:\|h_s^\eta-1\|_2^2\: \dd s\\
	&\qquad\le e^{-\frac{t}{t_\rel}}\:\left\| \frac{\eta}{\pi}-1\right\|_2^2+\:e^{-\frac{t}{t_\rel}}\: \left\|\frac{\eta}{\pi}-1 \right\|_2^2  \:\int_0^r \frac{1}{t_\rel}\, e^{-\frac{s}{t_\rel}}\: \dd s\\
	&\qquad\le 2  e^{-\frac{t}{t_\rel}}\:\left\| \frac{\eta}{\pi}-1\right\|_2^2,
	\end{align*}
	so that \cref{eq:first-integral} is bounded above by
	\begin{equation}\label{eq:first-integral2}
	\begin{split}
	\int_0^r \sum_{xy\in E} \cxy \frac{\pi(x)\pi(y)}{\pi(x)+\pi(y)}\left(h^{\eta}_s(x)-h^{\eta}_s(y) \right)^2 \Phi_{t-s}(x,y)\, \dd s\\\le\frac{ 4 \:c_{\rm ell}\: c}{n} e^{-\frac{t}{t_\rel}}\:\left\| \frac{\eta}{\pi}-1\right\|_2^2.
	\end{split}
	\end{equation}
	Collecting \cref{eq:second-integral,eq:first-integral2} and recalling that $r=t-C t_\rel$ we conclude that
	\begin{align}\label{eq:final-bound-Nt}
	&\cN_t(\eta)\le c_{\rm ell} \( \frac{ 4  c}{n} \exp\left\{-\frac{t}{t_\rel} \right\} +  e^{2C}\exp\left\{-\frac{2 t}{t_\rel} \right\}  \) \left\|\frac{\eta}{\pi}-1 \right\|_2^2 .
	\end{align}
	By combining \cref{eq:first-term,eq:final-bound-Nt} we finally obtain
	\begin{equation}
	\E^\avg_{\eta}\left[\left\|\frac{\eta_t}{\pi}-1 \right\|^2_2 \right]\leq	c_{\rm ell} \( \frac{ 4  c}{n} \exp\left\{-\frac{t}{t_\rel} \right\} +  (1+e^{2C})\exp\left\{-\frac{2 t}{t_\rel} \right\}  \) \left\|\frac{\eta}{\pi}-1 \right\|_2^2,
	\end{equation}
	and defining properly the constant $C'$ in the statement, we obtain the desired result.
\end{proof}
\subsection{Proofs of \cref{pr:no-cutoff-avg,th:mixing-avg}}\label{suse:proofs-avg}
\begin{proof}[Proof of \cref{pr:no-cutoff-avg}]
	The lower bound in \cref{eq:no-cutoff-avg} follows immediately by \cref{lemma:lower_bound_general}. Concerning the upper bound, recall the estimate in \cref{eq:518} in \cref{proposition:nash}. Then, by setting $t=b t_\rel$ for large enough $b$ and applying Jensen inequality, we  conclude the proof of \cref{eq:no-cutoff-avg} for $p=2$. The corresponding result for $p\in[1,2)$ follows again by Jensen inequality.
\end{proof}

\begin{proof}[Proof of \cref{th:mixing-avg}]
	As in the proof of \cref{pr:no-cutoff-avg}, the lower bound is an immediate consequence of \cref{lemma:lower_bound_general}. For the upper bound we will exploit \cref{lemma:upper-bound-avg-large}. Take $t=t^+(C)$ as in \cref{eq:def-t+t-} for some sufficiently large $C$ to be fixed later. Fixing now $s=\frac{C}{2}t_\rel$,  $k\to\infty$ ensures that $t-s> C t_\rel$ for large enough $n$; hence,  we can apply \cref{lemma:upper-bound-avg-large} in the time window $[s,t]$, namely,
	\begin{equation}\label{eq:stima-prova-cutoff-avg}
	\begin{split}
	&\E^\avg_\eta\[\left\|\frac{\eta_t}{\pi}-1 \right\|_2^2 \]=\E^\avg_\eta\[\E^\avg_\eta\[ \left\|\frac{\eta_t}{\pi}-1\right\|_2^2\: \bigg\rvert\: \eta_s \] \]\\
	&\qquad\le C'\: \( \frac{ 1}{n} \exp\left\{-\frac{t-s}{t_\rel} \right\} + \exp\left\{-\frac{2\ (t-s)}{t_\rel} \right\}  \) \E^\avg_\eta\[\left\|\frac{\eta_s}{\pi}-1 \right\|_2^2\]\ ,
	\end{split}
	\end{equation}
	where we used the Markov property and the estimate in  \cref{eq:lemma-long-times} for the process starting from $\eta_s$ and evolving for a time $t-s$. Now, if $C$ is large enough, we can use \cref{proposition:nash} to bound the expectation in the right-hand side of \cref{eq:stima-prova-cutoff-avg}, i.e., 
	\begin{equation}\label{eq:last-est}
	\E^\avg_\eta\[\left\|\frac{\eta_t}{\pi}-1 \right\|_2^2 \]\le  c\: C'\: \( \frac{ 1}{n} \exp\left\{-\frac{t}{t_\rel} \right\} + \exp\left\{-\frac{2t-s}{t_\rel} \right\}  \).
	\end{equation}
	Multiplying both sides in \cref{eq:last-est} by $k$, and substituting the values of $t=t_\mx + C t_\rel$ and $s=\frac{C}{2} t_\rel$,	 we obtain
	\begin{equation}\label{eq:last-est2}
	k\:\E^\avg_\eta\[\left\|\frac{\eta_t}{\pi}-1 \right\|_2^2 \]\le  c\: C' \( \frac{ \sqrt{k}}{n} e^{-C} +  e^{-\frac{3}{2}C} \).
	\end{equation}
	 The upper bound in \cref{eq:cutoff-avg} follows for all $p\in[1,2]$ by  applications of Jensen inequality.\end{proof}

We conclude this section by showing that when the assumption $k=O(n^2)$ in \cref{eq:hp-k}
is dropped, the arguments used so far show that a weaker version of \cref{th:cutoff} still holds.
\begin{proposition}[Pre-cutoff]\label{th:precutoff-avg}
	Consider a sequence of graphs and site-weights such that \cref{assumption:nash,assumption:unif_ellipticity} hold. Then, if $k/n^2\to\infty$, for all  $\delta\in(0,1)$,
	\begin{equation}\label{eq:precutoff-avg}
	\limsup_{n\to\infty} \sqrt{k}\: \sup_{\eta\in\Delta}\E_\eta^\avg\[\left\|\frac{\eta_{T^+}}{\pi}-1 \right\|_{p} \]\le\delta,\qquad
	\liminf_{n\to\infty} \sqrt{k}\: \sup_{\eta\in\Delta}\E_\eta^\avg\[\left\|\frac{\eta_{T^-}}{\pi}-1 \right\|_{p} \]\ge \frac{1}{\delta},
	\end{equation}
	for all $p\in[1,2]$, where 
	\begin{equation}
	T^+\coloneqq a\, \frac{t_\rel}{2}\log(k) + C t_\rel , \qquad T^-\coloneqq \frac{t_\rel}{2}\log(k) - C t_\rel,
	\end{equation}
	\begin{equation}
	a\coloneqq 2\,\frac{\log(k/n)}{\log(k)}\in[1,2],
	\end{equation}
	and  some $C=C(\delta)>0$. 
\end{proposition}
\begin{proof}
The upper bound (for $p=2$) follows directly by \cref{eq:last-est}. The lower bound (for $p=1$) follows by \cref{lemma:lower_bound_general}.
\end{proof}
Notice that, as soon as $k=\Omega(n^{2+\varepsilon})$ for some $\varepsilon>0$, then $T^+-T^-=\Theta(t_\rel \log(k) )$, namely, their first order terms do not coincide. On the other hand, \cref{th:precutoff-avg}  implies a cutoff also for all those $k=n^{2+o(1)}$, but with a larger  window having size $\omega(t_\rel)$. 

\section{Proofs  from \cref{sec:results_bin}}
\label{sec:proof_cutoff_bin}
In this section, we provide the proofs of the results from \cref{sec:results_bin}. More specifically, in \cref{sec:upper_bound-mult,sec:upper_bound-non-mult,sec:lower_bound-bin} below we derive some intermediate results needed for such proofs, which are then presented in \cref{sec:proofs-bin} below.	
In particular, in \cref{sec:upper_bound-mult} we provide an upper bound for the \ac{TV}-distance of the Binomial Splitting process when the initial distribution of particles is Multinomial in terms of the $L^2$-distance of the Averaging process. The crucial ingredient in this step is the Multinomial intertwining in \cref{pr:mult-inter}.
In \cref{sec:upper_bound-non-mult} we extend this bound to all possible initial distribution by means of ``multi-colored'' auxiliary processes defined therein. Finally, in \cref{sec:lower_bound-bin} we use Wilson's method, originally introduced in \cite{Wilson}, to show a matching lower bound.
\subsection{Upper bound for Multinomial initial distributions}\label{sec:upper_bound-mult}
\begin{lemma}\label{lemma:cutoff-bin-ub}
	For all $k\in\N$ and $\eta\in\Delta$, recall the definition of $\mu_{k,\eta}$ as the Multinomial distribution with parameters $(k,\eta)$. Then,
	\begin{equation}
	\|\mu_{k,\eta}\, \cS^{\bin(k)}_t-\mu_{k,\pi}\|_{\TV} \le \sqrt{ek\:\E^{\avg}_\eta\[\left\|\frac{\eta_t}{\pi}-1 \right\|^2_2 \]},\qquad t\ge 0\, .
	\end{equation}
\end{lemma}
We defer its proof after  the following preliminary result, in which we derive an upper bound on the \ac{TV} distance between two Multinomial distributions with the same sample size.
\begin{lemma}\label{lemma:distance-multinomial}
	For all $k\in\N$ and $\eta,\pi\in\Delta$, 
	\begin{equation}
	\|\mu_{k,\eta}-\mu_{k,\pi} \|_{\TV}\le 1 \:\wedge\: \left\|\frac{\mu_{k,\eta}}{\mu_{k,\pi}}-1 \right\|_{L^2(\Omega_k,\mu_{k,\pi})}\le \sqrt{ek}\: \left\|\frac{\eta}{\pi}-1\right\|_2\, .
	\end{equation}
\end{lemma}
\begin{proof}
	Jensen inequality yields the first inequality. As for the second one, by simple manipulations with Multinomial distributions, we have
	\begin{align}
	\begin{aligned}
	\left\|\frac{\mu_{k,\eta}}{\mu_{k,\pi}}-1 \right\|_{L^2(\Omega_k,\mu_{k,\pi})}^2&=  \sum_{\xi \in \Omega_k}\mu_{k,\pi}(\xi) \(\(\frac{\mu_{k,\eta}(\xi)}{\mu_{k,\pi}(\xi)}\)^2-1\)\\
	&=\(\sum_{\xi \in \Omega_k}\mu_{k,\pi}(\xi) \prod_{x\in V}\(\frac{\eta(x)}{\pi(x)}\)^{2 \xix} \)-1\\
	&= \(\sum_{\bx \in V^k}\pi(\bx) \big(D(\bx,\eta)\big)^2  \)-1\\
	&= \(\sum_{x \in V}\pi(x) \big(D(x,\eta)\big)^2  \)^k-1\\
	&=\(1+\left\|\frac{\eta}{\pi}-1\right\|^2_2 \)^k-1.
	\end{aligned}
	\end{align}
(Note that for the third step, we just used that, for any labeled configuration $\bx\in V^k$ and its corresponding unlabeled one $\xi\in\Omega_k$, $D(\bx,\eta)=\prod_{x\in V}\frac{\eta(x)}{\pi(x)}^{\xi(x)}$ holds for all $\eta\in \Delta$.)
	The  inequality
	$	1\:\wedge\:\[(1+a)^k-1\]\le ek a$, $a\ge 0$,	yields the desired result.
\end{proof}
\begin{proof}[Proof of \cref{lemma:cutoff-bin-ub}]
	By using the intertwining relation in \cref{pr:mult-inter}, we have
	\begin{align}
	\begin{split}
	\left\| \mu_{k,\eta} \cS_t^{\bin(k)} - \mu_{k,\pi}\right\|_\TV\coloneqq&\ \max_{A \subseteq \Omega_k} \left|\varLambda_k  \cS_t^{\bin(k)} \ind_A(\eta)- \varLambda_k \ind_A(\pi) \right|\\
	=&\ \max_{A \subseteq \Omega_k} \left|\cS_t^{\avg}\varLambda_k   \ind_A(\eta)- \varLambda_k \ind_A(\pi) \right|.
	\end{split}
	\end{align}
	By moving the maximum within the expectation, the right-hand side above is further bounded by
	\begin{equation}
	\left\| \mu_{k,\eta} \cS_t^{\bin(k)} - \mu_{k,\pi}\right\|_\TV\le \E^\avg_\eta\big[ \left\|\mu_{k,\eta_t} - \mu_{k,\pi}\right\|_{\TV} \big].
	\end{equation}
	Finally, applying \cref{lemma:distance-multinomial} and Jensen inequality concludes the proof.
\end{proof}

\subsection{Multicolored processes and upper bound for general initial conditions}\label{sec:upper_bound-non-mult}
The upper bound in \cref{lemma:cutoff-bin-ub} in the previous section was derived only for particle systems initialized according to multinomial distributions;  as a particular case, setting $\eta=\delta_x \in \Delta$, \cref{lemma:cutoff-bin-ub} holds true for $\mu_{k,\eta}=\mu_{k,\delta_x}$, i.e., the Dirac measure on the configuration with all $k$ particles  piled at $x \in V$. Of course, the same upper bound carries over to arbitrary convex combinations of multinomial distributions: for all probability measures $\nu$ on $\Delta$, 
\begin{equation}
\left\|\int_\Delta \nu(\dd \eta)\, \mu_{k,\eta}\, \cS^{\bin(k)}_t-\mu_{k,\pi} \right\|_\TV\le \sqrt{e k}\, \E^\avg_\nu\left[\left\| \frac{\eta_t}{\pi}-1\right\|_2 \right],\qquad t\ge 0.	
\end{equation}
However, while De Finetti's theorem ensures that such convex combinations would exhaust the space of probability measures on $\Omega_k$ if $V$ were infinite, this is far from being true when $V$ is finite (see, e.g., \cite{diaconis_finite_1980, aldous_exchangeability1983}). In fact, it is simple to check that there exists a constant $\gamma \in (0,1)$ such that
\begin{equation}
\sup_{\mu}\inf_{\nu}\left\|\mu-\int_\Delta \nu(\dd \eta)\, \mu_{k,\eta} \right\|_\TV\ge \gamma
\end{equation}
holds for all $n$ and $k\in \N$ large enough, where the above supremum runs over all probability measures on $\Omega_k$. Hence, multinomial initial distributions do not suffice to approximate via convex combinations all initial conditions.

In order to bypass this obstacle and derive an upper bound for all initial conditions, we introduce below what we call \emph{multicolored Averaging} and \emph{Binomial Splitting processes}, and prove a corresponding  intertwining relation between them.

Let us start by assigning to each site $z \in V$ a different color. For each  $z \in V$, consider a probability measure $\eta^{(z)}\in \Delta$ on $V$ equipped with the corresponding color. We now construct the multicolored Averaging $\(\vec \eta_t\)_{t\ge 0} = ((\eta^{(z)}_t)_{z\in V})_{t\ge 0} \subseteq \Delta^V$ started from $\vec \eta=(\eta^{(z)})_{z\in V}$ by means of the following grand coupling: at time $t=0$, each $z$-coordinate $\eta^{(z)}_0$ is set equal to $\eta^{(z)}\in \Delta$; then, at independent exponential times of rates $(\cxy)_{xy \in E}$, let all coordinates $(\eta^{(z)}_t)_{t\ge 0}$, $z\in V$,  thermalize simultaneously their values at the same edge. In other words, each colored coordinate evolves as the Averaging process, with the constraint that all share the same random sequence of edge updates. More precisely, recalling the definition in \cref{eq:def-Lavg}, the generator of this colored process writes as
\begin{equation}
\vec\cL^\avg f\coloneqq \sum_{xy\in E}\cxy\( \otimes_{z\in V}\(\cP^\avg_{xy}-\identity\)\)f, 
\end{equation} 
for all $f \in \cC(\Delta^V)=\otimes_{z\in V}\,\cC(\Delta)$.

An entirely analogous construction carries over for the multicolored Binomial Splitting started from the configuration $\xi \in \Omega_k$: at time $t=0$, for all sites $z\in V$, the $\xiz$ particles at $z \in V$ are given the color assigned to $z \in V$; then, following the same edge updates, each colored particle system performs the Binomial Splitting dynamics. Hence, we write $(\vec \xi_t)_{t\ge 0}=((\xi^{(z)}_t)_{z\in V})_{t \ge 0}\subseteq \prod_{z\in V}\Omega_{\xiz}$ to indicate the vector of such colored Binomial Splitting process started from $\xi \in \Omega_k$, with generator given by
\begin{equation}
\vec \cL^{\,\bin(\xi)}f\coloneqq \sum_{xy \in E}\cxy\(\otimes_{z\in V}\(\cP^{\bin(\xiz)}_{xy}-\identity\)\)f
\end{equation}
for all $f: \prod_{z\in V}\Omega_{\xiz}\to \R$.
Notice that choosing   $\xi\in \Omega_k$ with
\begin{equation}
w\mapsto\xi(w)\coloneqq \begin{cases}
k &\text{if}\ w=z\\
0&\text{otherwise}
\end{cases}
\end{equation}
as initial configuration,
the projection of $(\vec \xi_t)_{t\ge0}$ on its $z$-coordinate corresponds to the $\bin(k)$ system started with all $k$ particles at $z\in V$.

Having introduced such multicolored processes, the following intertwining relation generalizes \cref{pr:mult-inter} to such processes.
\begin{proposition}[Intertwining for the multicolored processes]\label{pr:mult-inter-color}
	For all $k \in \N$, $\xi \in \Omega_k$ and $t\ge 0$,	 
	\begin{equation}\label{eq:intertwining-colored}
	\vec\cS^\avg_t \vec \varLambda_\xi f= \vec \varLambda_\xi \vec \cS^{\,\bin(\xi)}_t f,\qquad f : \prod_{z\in V} \Omega_{\xiz}\to \R,
	\end{equation}
	where 
	\begin{equation}
	\vec \varLambda_\xi f(\vec \eta)\coloneqq \E_{\(\otimes_{z\in V}\,\mu_{\xiz,\eta^{(z)}}\)}\left[f \right],
	\end{equation}
	and
	\begin{equation}
	\vec \cS^\avg_t\coloneqq e^{t\vec \cL^\avg}\qquad \text{and}\qquad \vec \cS^{\,\bin(\xi)}_t\coloneqq e^{t \vec \cL^{\,\bin(\xi)}},\qquad t \ge 0.
	\end{equation}
\end{proposition}
\begin{proof}
	As in the proof of \cref{pr:mult-inter}, checking identity \cref{eq:intertwining-colored} for the corresponding bounded generators, by linearity, it suffices to verify that
	\begin{equation}\label{eq:intertwining-colored-p}
	\(\otimes_{z\in V}\(\cP^\avg_{xy}-\identity\)\)\vec \varLambda_\xi f= \vec \varLambda_\xi \(\otimes_{z\in V}\( \cP^{\bin(\xiz)}_{xy}-\identity\)\)f	
	\end{equation}
	holds for all $xy\in E$, $\xi \in \Omega_k$ and $f: \prod_{z\in V}\Omega_{\xiz}\to \R$. The rest of the argument goes on as in the proof of \cref{pr:mult-inter} due to the product structure of the operators involved.
\end{proof}
In view of the above intertwining, we are now ready to extend the upper bound in \cref{lemma:cutoff-bin-ub} to all initial conditions.
\begin{lemma}\label{lemma:ub-general-initial}
	For all $k\in \N$ and $t \ge0$, 
	\begin{equation}
	\sup_{\mu}\left\|\mu\, \cS_t^{\bin(k)}-\mu_{k,\pi} \right\|_\TV\le \sqrt{ek\,\sup_{\eta \in \Delta}\E^\avg_\eta\left[\left\|\frac{\eta_t}{\pi}-1\right\|_2^2 \right]},
	\end{equation}
	where the supremum on the left-hand side runs over all probability measures on $\Omega_k$.
\end{lemma}
\begin{proof}
	First observe that, since all measures $\mu$ write as convex combinations of Dirac measures $(\delta_\xi)_{\xi \in \Omega_k}$,
	\begin{equation}\label{eq:color-1}
	\sup_{\mu}\left\|\mu\, \cS_t^{\bin(k)}-\mu_{k,\pi} \right\|_\TV\le \sup_{\xi \in \Omega_k} \left\|\delta_\xi\, \cS^{\bin(k)}_t-\mu_{k,\pi} \right\|_{\TV}.
	\end{equation} 
	Then, since $\(\xi_t\)_{t\ge 0}$ started from $\xi \in \Omega_k$ can be obtained as a projection of the multicolored process $(\vec \xi_t)_{t\ge 0}$ also started from $\xi \in \Omega_k$, we have
	\begin{equation}\label{eq:color-2}
	\left\|\delta_\xi\, \cS^{\bin(k)}_t-\mu_{k,\pi} \right\|_{\TV}\le \left\|\(\otimes_{z\in V}\,\mu_{\xiz,\delta_z}\)\vec \cS^{\,\bin(\xi)}_t-\(\otimes_{z\in V}\,\mu_{\xiz,\pi}\) \right\|_\TV.
	\end{equation}
	Arguing as in the proof of \cref{lemma:cutoff-bin-ub}, by means of the intertwining relation in \cref{pr:mult-inter-color},  it follows that
	\begin{align}\label{eq:color-3}
	\nonumber
	&\left\|\(\otimes_{z\in V}\,\mu_{\xiz,\delta_z}\)\vec \cS^{\,\bin(\xi)}_t-\(\otimes_{z\in V}\,\mu_{\xiz,\pi}\)  \right\|_\TV\\
	\nonumber
	&\qquad\le \sup_{\vec \eta \in \Delta^V}\E_{\vec \eta}^{\vec \avg}\left[\left\|\(\otimes_{z\in V}\,\mu_{\xiz,\eta^{(z)}_t}\)-\(\otimes_{z\in V}\,\mu_{\xiz,\pi}\)  \right\|_\TV \right]\\
	&\qquad\le \sup_{\vec \eta \in \Delta^V} \E^{\vec\avg}_{\vec \eta}\[\(1 \wedge \sum_{z\in V}\left\|\frac{\mu_{\xiz,\eta^{(z)}_t}}{\mu_{\xiz,\pi}}-1 \right\|^2_{L^2(\Omega_\xiz,\mu_{\xiz,\pi})}\)^{\frac{1}{2}}\],
	\end{align}
	where in the last inequality we employed \cite[Proposition 7]{lubetzky_sly_2014} to bound from above the \ac{TV} distance between two product measures. (Here, $\E^{\vec \avg}_{\vec \eta}$ denotes expectation with respect to the law of the multicolored Averaging started from $\vec \eta \in \Delta^V$.) By Jensen inequality and the observation that each colored marginal of $(\vec \eta_t)_{t\ge 0}\subseteq \Delta^V$ evolves like the Averaging process, \cref{lemma:distance-multinomial} further yields
	\begin{align}\label{eq:color-4}
	\nonumber
	&\sup_{\vec \eta \in \Delta^V} \E^{\vec\avg}_{\vec \eta}\[\(1\wedge \sum_{z\in V}\left\|\frac{\mu_{\xiz,\eta^{(z)}_t}}{\mu_{\xiz,\pi}}-1 \right\|^2_{L^2(\Omega_\xiz,\mu_{\xiz,\pi})}\)^{\frac{1}{2}}\]\\
	&\qquad\le 	\sup_{\eta \in \Delta}\sqrt{\sum_{z \in V}e\,\xiz\, \E^\avg_\eta\[\left\|\frac{\eta_t}{\pi}-1 \right\|_2^2\] }
	=	\sup_{\eta \in \Delta}\sqrt{ek\,	 \E^\avg_\eta\[\left\|\frac{\eta_t}{\pi}-1 \right\|_2^2\] }.	
	\end{align}
	Combining the inequalities in \cref{eq:color-1,eq:color-2,eq:color-3,eq:color-4} concludes the proof. 
\end{proof}
\subsection{Lower bound}\label{sec:lower_bound-bin}
The lower bound proof is based on Wilson's method (see, e.g., \cite[Theorem 13.28]{levin2017markov}). Roughly speaking, the method amounts in exhibiting a \emph{distinguishing statistics}, $F$, whose knowledge of means and variances in and out of equilibrium suffices to obtain a lower bound for the \ac{TV}-distance to equilibrium. Further, as we will now explain, such a distinguishing statistics is built from an eigenfunction of the Markov generator of the single-particle system.

Let $\psi: V \to \R$ be an eigenfunction of $-L^{\bin(1)}$ associated to the spectral gap and such that $\|\psi \|_2=1$; in particular, we have	
\begin{equation}\label{eq:eigenfunction_normalization}
\sum_{x \in V}\pi(x)\, \psi(x)=0\qquad \text{and}\qquad \sum_{x \in V} \pi(x)\left(\psi(x) \right)^2 = 1.
\end{equation}
Wilson's method will be applied to the observable $F:\Omega_k\to \R$ defined as
\begin{equation}\label{eq:observable_wilson}
	F(\xi)\coloneqq \sum_{x \in V} \psi(x)\, \xi(x)\ .
\end{equation}
To the purpose of analyzing $F$, recall that the mean and covariance of a Multinomial distribution are given by
\begin{equation}\label{eq:mean-cov-mult}
\begin{split}
\E_{\mu_{k,\eta}}[\xi(x)]&=k\eta(x),\\ {\rm Cov}_{\mu_{k,\eta}}\big(\xi(x),\xi(y)\big)&=-k\eta(x)\eta(y)\ind_{\{x\neq y\}}+k\eta(x)(1-\eta(x))\ind_{\{x=y\}}.
\end{split}
\end{equation}
Recall that $\pl \cdot,\cdot\pr$ denotes the inner product in $\cH=L^2(V,\pi)$.
\begin{lemma}[Mean and variance at equilibrium]\label{le:wilson}
	For every sequence $k=k_n$,
	\begin{equation}\label{eq:mean-var_at_equilibrium}
	\E_{\mu_{k,\pi}}\left[F \right]=0,\qquad \text{and} \qquad
	\var_{\mu_{k,\pi}}(F)=k\ .	
	\end{equation}
\end{lemma}
\begin{proof}
	The proof of \cref{le:wilson} follows  by the definition of $F$ in \cref{eq:observable_wilson},  as well as the identities in \cref{eq:eigenfunction_normalization,eq:mean-cov-mult}.
\end{proof}
\begin{lemma}[Mean and variance out of equilibrium]\label{le:wilson2}
	For every sequence $k=k_n$, $\eta\in\Delta$ and $t\ge C t_\rel$, where $C$ is as in \cref{proposition:nash},
	\begin{equation}\label{eq:expectation_out_equilibrium}
	\E_{\mu_{k,\eta}\,\cS_t^{\bin(k)}}\left[F \right] =  k e^{-\frac{t}{t_\rel}}\pl \psi,D(\cdot,\eta)\pr,	
	\end{equation}
	and
	\begin{equation}\label{eq:upper_bound_variance_out-lemma}
	\var_{ \mu_{k,\eta}\,\cS_t^{\bin(k)}  }\left(F \right)\leq C'\[k+  \frac{k^2}{n}\(\exp\left\{-\frac{2 t}{t_\rel}\right\} \|D(\cdot,\eta)\|_\infty^2 + \exp\left\{-\frac{ t}{t_\rel}\right\}\)\],
	\end{equation}
	for some absolute constant $C'>1$.
\end{lemma}
\begin{proof}
	To prove the identity in \cref{eq:expectation_out_equilibrium} we use the self-duality in \cref{proposition:self-duality} with $k=1,2$, the first identity in \cref{eq:mean-cov-mult}, self-adjointness of $S_t^{\bin(1)}$ and the fact that $\psi$ is an eigenfunction:
	\begin{equation}\label{eq:612}
	\begin{split}
	\E_{\mu_{k,\eta}\,\cS_t^{\bin(k)}}\left[F \right] &= \sum_{x\in V}\pi(x)\,\psi(x)\,\E^{\bin(k)}_{\mu_{k,\eta}}\[\frac{\xi_t(x)}{\pi(x)} \]\\
	&=\sum_{x\in V} \pi(x)\,\psi(x)\,\E_{\mu_{k,\eta}}\[S_t^{\bin(1)}\( \frac{\xi(\cdot)}{\pi(\cdot)} \)(x)\ \]\\
	&=k\sum_{x\in V}\pi(x)\,\psi(x)\, S_t^{\bin(1)}\(D(\cdot,\eta)\)(x)\\
	&=k e^{-\frac{t}{t_\rel}}\sum_{x\in V}\pi(x)\,\psi(x)\,D(x,\eta).
	\end{split}
	\end{equation}
	For what concerns the proof of \cref{eq:upper_bound_variance_out-lemma}, we have
	\begin{align*}
	\var&_{\mu_{k,\eta}\,\cS_t^{\bin(k)}}\left(F \right)=\sum_{x\in V}\psi(x)^2\,\E^{\bin(k)}_{\mu_{k,\eta}}[\xi_t(x)]\ +\\
	&+\sum_{x,y\in V}\psi(x)\psi(y)\left\{\E^{\bin(k)}_{\mu_{k,\eta}}\[\xi_t(x)\(\xi_t(y)-\ind_{x=y} \) \]-\E^{\bin(k)}_{\mu_{k,\eta}}[\xi_t(x)]\E^{\bin(k)}_{\mu_{k,\eta}}[\xi_t(y)]\right\}.
	\end{align*}
	As already noticed in \cref{eq:612}, we have, for all $x\in V$,
	\begin{equation}
	\E^{\bin(k)}_{\mu_{k,\eta}}[\xi_t(x)]
	=k\,\pi(x)\,S_t^{\bin(1)}(D(\cdot,\eta))(x);
	\end{equation}
	analogously, it follows that, for all $x,y\in V$,	
	\begin{align}
	\E^{\bin(k)}_{\mu_{k,\eta}}[\xi_t(x)]\E^{\bin(k)}_{\mu_{k,\eta}}[\xi_t(y)]&=
	k^2\,\pi(x)\pi(y)\, S_t^{\bin(1)\otimes \bin(1)} \big(D(\cdot,\eta) \big)(x,y),\\
	\E^{\bin(k)}_{\mu_{k,\eta}}\[\xi_t(x)\(\xi_t(y)-\ind_{x=y} \) \]&=
	k(k-1)\,\pi(x)\pi(y)\, S_t^{\bin(2)}\big(D(\cdot,\eta) \big)(x,y).
	\end{align}
	Therefore
	\begin{equation}\label{eq:first-display-var}
	\begin{split}
	&\var_{ \mu_{k,\eta}\cS_t^{\bin(k)}  }\left(F \right)=\\
	&= k\sum_{x\in V}\pi(x)\,\psi(x)^2\,S_t^{\bin(1)}(D(\cdot,\eta))(x) -k\(\sum_{x\in V}\pi(x)\,\psi(x)\,S_t^{\bin(1)}(D(\cdot,\eta))(x)\)^2 \\
	&+k(k-1)\sum_{x,y\in V}\pi(x)\pi(y)\,\psi(x)\psi(y) \(S_t^{\bin(2)} - S_t^{\bin(1)\otimes \bin(1)}\)\big(D(\cdot,\eta) \big)(x,y)\\
	&= k\, \var_{\eta\, S_t^{\bin(1)}}\( \psi\)+\\
	&+k(k-1)\sum_{x,y\in V}\pi(x)\pi(y)\,\psi(x)\psi(y) \(S_t^{\bin(2)} - S_t^{\bin(1)\otimes \bin(1)}\)\big(D(\cdot,\eta) \big)(x,y).
	\end{split}
	\end{equation}
	By \cref{assumption:nash}, the choice of $t>C t_\rel$ and $\| \psi\|_2^2=1$, recalling \cref{eq:nash-decay}, we obtain 
	\begin{equation}\label{eq:var-first-term}
	\var_{\eta\,	 S_t^{\bin(1)}}\( \psi\)\le\pl \psi^2, S^{\bin(1)}_tD(\cdot,\eta)\pr=\pl\psi^2, h^\eta_t\pr  \le c,
	\end{equation}
	for some constant $c>1$ independent of $n$. On the other hand, by the integration by parts formula, we rewrite
	\begin{equation*}
	\sum_{x,y\in V}\pi(x)\pi(y)\,\psi(x)\psi(y) \(S_t^{\bin(2)} - S_t^{\bin(1)\otimes \bin(1)}\)\big(D(\cdot,\eta) \big)(x,y)
	\end{equation*}
	as ($L^{(2)}-L^{(1,1)}:= L^{\bin(2)}-L^{\bin(1)\otimes \bin(1)}$)
\begin{equation*}
	\begin{split}
	&\int_0^t \sum_{x,y \in V}\pi(x)\pi(y)\left(S^{\bin(2)}_{t-s}\left(L^{(2)}-L^{(1,1)} \right)S_s^{\bin(1)\otimes \bin(1)} \right)(\psi\otimes\psi)(x,y) \, D((x,y),\eta)\, \dd s\\
	&=  \int_0^t \sum_{x,y \in V}\pi(x)\pi(y)\, e^{-\frac{2s}{t_\rel}}  (L^{(2)}-L^{(1,1)})(\psi\otimes \psi)(x,y)\, S^{\bin(2)}_{t-s}D(\cdot,\eta)(x,y)\, \dd s\ ,
	\end{split}
	\end{equation*}
where the last identity is an immediate consequence of self-adjointness of the $\bin(2)$-semigroup and the fact that $\psi$ is eigenfunction for $\bin(1)$. By performing a change of  variables and using the explicit computation in \cref{eq:strange_object},
	we further obtain
	\begin{equation}\label{eq:second-term-var}
	\begin{split}
	&\sum_{x,y\in V}\pi(x)\pi(y)\,\psi(x)\psi(y) \(S_t^{\bin(2)} - S_t^{\bin(1)\otimes \bin(1)}\)\big(D(\cdot,\eta) \big)(x,y)\\
	&= e^{-\frac{2t}{t_\rel}}\sum_{xy\in E}\cxy\frac{\pi(x)\pi(y)}{\pi(x)+\pi(y)}\left(\psi(x)-\psi(y) \right)^2\int_0^t e^{\frac{2s}{t_\rel}}\frac{\pi(x)\pi(y)}{\pi(x)+\pi(y)}\Psi_s(x,y)\, \dd s\  ,
	\end{split}
	\end{equation}	
	where, for notational convenience,  we write, for all $s\ge 0$ and $x, y \in V$, 
	\begin{align*}
	&\Psi_s(x,y)\coloneqq \E^{\avg}_\eta\[\(\frac{\eta_s(x)}{\pi(x)}-\frac{\eta_s(y)}{\pi(y)}\)^2 \]\\
	=&\sum_{z,w\in V}\eta(z)\eta(w)\[\frac{p_s^{\bin(2)}((z,w),(x,x))}{\pi(x)^2}+\frac{p_s^{\bin(2)}((z,w),(y,y))}{\pi(y)^2}-2\,\frac{p_s^{\bin(2)}((z,w),(x,y))}{\pi(x)\pi(y)}\].
	\end{align*}
	Let us now provide two upper bounds for $\Psi_s(x,y)$ depending on the value of $s\ge 0$. On the one hand, for all $s\ge 0$ and $x,y\in V$, by $(a-b)^2\le 2a^2+2b^2$ and the duality in \cref{pr:dualities} (in \cref{eq:dual_rel}) with $k=2$, we have
	\begin{equation}
	\Psi_s(x,y)\le 2 \E^{\avg}_\eta\[\(\frac{\eta_s(x)}{\pi(x)}\)^2 \]+2 \E^{\avg}_\eta\[\(\frac{\eta_s(y)}{\pi(y)}\)^2 \]\le 4\max_{z\in V} \(\frac{\eta(z)}{\pi(z)}\)^2=4\left\|D(\cdot,\eta) \right\|_\infty^2.
	\end{equation}
	On the other hand, by H\"older inequality and $\eta \in \Delta$,    \cref{proposition:nash} ensures that  
	\begin{equation}
	\Psi_s(x,y)\le 4 \max_{x,y,z,w\in V} \left|\frac{p_s^{\bin(2)}((x,y),(z,w))}{\pi(z)\pi(w)}-1\right|\le 4 c e^{-\frac{s}{t_\rel}},\qquad s\ge Ct_\rel.
	\end{equation} 
	As done  in the proof of \cref{lemma:upper-bound-avg-large}, we split the integral in \cref{eq:second-term-var} in two, obtaining
	\begin{align}\label{eq:est-integral}
	\nonumber\int_0^t \exp\left\{\frac{2 s}{t_\rel}\right\}&\frac{\pi(x)\pi(y)}{\pi(x)+\pi(y)}\Psi_s(x,y)\, \dd s \\
	\nonumber &\le \frac{4c_{\rm ell}}{n}\|D(\cdot,\eta)\|_\infty^2\int_0^{C t_\rel}\exp\left\{\frac{2s}{t_\rel} \right\}\dd s + \frac{4c\:c_{\rm ell}}{n}\int_{Ct_\rel}^t  \exp\left\{\frac{s}{t_\rel} \right\}  \dd s\\
\nonumber	&\le \frac{t_\rel}{n} \(2c_{\rm ell}e^{2C}\|D(\cdot,\eta)\|_\infty^2 + 4c\, c_{\rm ell} e^{\frac{t}{t_\rel}}\)\\
	&\le \tilde C \frac{t_\rel }{n}\(\|D(\cdot,\eta)\|_\infty^2 + e^{\frac{t}{t_\rel}} \),
	\end{align}
	where we absorbed all constants in the quantity $\tilde C>0$. Plugging \cref{eq:est-integral} into \cref{eq:second-term-var} and noticing that by our choice of $\psi$ we have $\cE_{\bin(1)}(\psi)=\frac{1}{t_\rel}$, we obtain
	\begin{equation}\label{eq:est-var-third}
	\begin{split}
	\sum_{x,y\in V}\pi(x)\pi(y)\,\psi(x)\psi(y) \(S_t^{\bin(2)} - S_t^{\bin(1)\otimes \bin(1)}\)\big(D(\cdot,\eta) \big)(x,y)\\
	\le\frac{\tilde C}{n}\(\exp\left\{-\frac{2 t}{t_\rel}\right\} \|D(\cdot,\eta)\|_\infty^2 + \exp\left\{-\frac{ t}{t_\rel}\right\}\).
	\end{split}
	\end{equation}
	Collecting the estimates in \cref{eq:est-var-third,eq:var-first-term}, and going back to \cref{eq:first-display-var} we finally obtain \cref{eq:upper_bound_variance_out-lemma}.
\end{proof}
The next proposition is a straightforward application of  \cite[Proposition 7.9]{levin2017markov}  using the estimates in \cref{le:wilson,le:wilson2}.
\begin{proposition}\label{pr:wilson-final}
	For all $k\in \N$, $\eta \in \Delta$ and $t\ge C t_\rel$,
	\begin{equation}
	\left\|\mu_{k,\eta}\:\cS^{\bin(k)}_t-\mu_{k,\pi}\right\|_\TV \geq 1 - \frac{8}{a_t(k,\eta)}\ , 
	\end{equation}
	where
	\begin{equation}
	a_t(k,\eta)\coloneqq \frac{k \pl \psi,D(\cdot,\eta)\pr^2}{1+  \frac{k}{n}\(\|D(\cdot,\eta)\|_\infty^2 + \exp\left\{\frac{ t}{t_\rel}\right\}\)}\ ,
	\end{equation}
	where $\psi:V\to\R$ is an eigenfunction of $-L^{\bin(1)}$ associated to the eigenvalue $\gap$ such that $\|\psi \|_2=1$.
\end{proposition}

\subsection{Proofs of \cref{pr:no-cutoff,th:cutoff}}\label{sec:proofs-bin}
\begin{proof}[Proof of \cref{pr:no-cutoff}]
	The upper bound in \cref{pr:no-cutoff} follows at once from \cref{lemma:ub-general-initial} and the upper bound in \cref{pr:no-cutoff-avg}.
	Concerning the lower bound, by the injectivity of the operator $J_k$ in \cref{eq:def-Jk} and the relation in \cref{eq:intertwining-Jk}, a projection argument ensures that
	\begin{equation}\label{eq:641}
	\left\| \mu_{k,\eta}\: \cS_t^{\bin(k)}-\mu_{k,\pi}\right\|_{\TV}\ge \left\| \eta S_t^{\bin(1)}-\pi \right\|_{\TV}.
	\end{equation}
	Thus, the second inequality in \cref{eq:pr-lb-w2} concludes the proof.
\end{proof}

\begin{proof}[Proof of \cref{th:cutoff}]
	The upper bound in \cref{th:cutoff} follows by \cref{lemma:ub-general-initial} and the upper bound in \cref{th:mixing-avg}. For what concerns the lower bound, we exploit  \cref{pr:wilson-final} by showing that $\sup_{\eta\in\Delta}a_{t^-}(k,\eta)$, with $t^-=t_\mx - t_\w(C)$, is increasing to infinity as $C$ grows. Notice that, with this choice of $t=t^-$, we have
	\begin{equation}\label{eq:at-}
	a_{t^-}(k,\eta)= \frac{k \pl \psi,D(\cdot,\eta)\pr^2}{1+  \frac{k}{n}\(\|D(\cdot,\eta)\|_\infty^2 + e^{-C}\sqrt{k} \)}\ .
	\end{equation}
	Concerning the numerator in \cref{eq:at-} we have
	\begin{equation}\label{eq:boh}
	\sup_{\eta\in\Delta} |\pl \psi,D(\cdot,\eta)\pr|\ge \sup_{\substack{
			\eta\in\Delta\\
			\| D(\cdot,\eta)\|_\infty\le 2}} |\pl \psi,D(\cdot,\eta)\pr|\geq  \sup_{\substack{
			\left\|\varphi\right\|_\infty \leq 2\\
			\left\|\varphi\right\|_{1,+}=1}}\langle \psi_+, \varphi\rangle   \vee  \sup_{\substack{
			\left\|\varphi\right\|_\infty \leq 2\\
			\left\| \varphi \right\|_{1,-}=1}} \langle \psi_-, \varphi\rangle \ ,
	\end{equation}
	where
	\begin{align*}
	V_+\coloneqq&\ \left\{x \in V: \psi_+(x)\coloneqq\max\left\{\psi(x),\, 0\right\}\geq 0\right\}\\
	V_-\coloneqq&\ \left\{x \in V: \psi_-(x)\coloneqq\max\left\{-\psi(x),\, 0\right\} >0\right\} ,
	\end{align*}  and	
	\begin{equation}
	\left\|\varphi\right\|_{1,\pm}\coloneqq \sum_{x \in V_\pm}\pi(x)\left| \varphi(x)\right|\ .
	\end{equation}
	Let us define $\pi(V_\pm)\coloneqq \sum_{x \in V_\pm}\pi(x)$; then,   $\pi(V_-)=1-\pi(V_+)$ and, since $\langle \psi, 1 \rangle = 0$ and by \cref{assumption:unif_ellipticity}, $\pi(V_\pm) \in \left[\frac{1}{c_{\rm ell}n},1-\frac{1}{c_{\rm ell}n} \right] \subseteq (0,1)$. Hence, for 
	\begin{equation}
	\varphi_\pm(x)\coloneqq \frac{\ind_{ V_\pm}(x)}{\pi(V_\pm)}\ ,\qquad x \in V\ ,
	\end{equation}
	we have, by definition, $\left\| \varphi_\pm\right\|_{1,\pm}=1$; moreover,  by the fact that $\pi(V_-)=1-\pi(V_+)$,  for all $n \in \N$, we have
	\begin{equation}\label{eq:boh2}
	1 \leq 		\left\| \varphi_+\right\|_\infty \wedge \left\| \varphi_-\right\|_\infty \leq 2\ .
	\end{equation}
	Hence, for each $n \in \N$, at least  one between $\varphi_+$ and $\varphi_-$ satisfies the constraints on the right-hand side in \cref{eq:boh}; therefore, since $\langle \psi_+,1\rangle=\langle \psi_-,1\rangle = \frac{1}{2}\left\| \psi\right\|_1$, this yields
	\begin{equation}\label{eq:fin-est-at-num}
	\sup_{\substack{\eta\in\Delta\\\left\| D(\cdot,\eta)\right\|_\infty\leq 2}} \left| \pl\psi, D(\cdot,\eta) \pr \right| \geq  \sup_{\substack{
			\left\|\varphi\right\|_\infty \leq 2\\
			\left\|\varphi\right\|_{1,+}=1}}\langle \psi_+, \varphi\rangle   \vee  \sup_{\substack{
			\left\|\varphi\right\|_\infty \leq 2\\
			\left\| \varphi \right\|_{1,-}=1}} \langle \psi_-, \varphi\rangle \geq \frac{1}{2} \left\|\psi \right\|_1\ ,
	\end{equation}
where the first inequality follows from \cref{eq:boh} and the second one follows by the lower bound in \cref{eq:boh2}.
	As a consequence, we obtain 
	\begin{equation}
	\sup_{\eta\in \Delta}a_{t^-}(k,\eta)\ge \sup_{\substack{\eta\in\Delta\\\|D(\cdot,\eta)\|_\infty\le 2 }}\frac{k\, \pl \psi,D(\cdot,\eta)\pr^2}{1+  \frac{k}{n}\(\|D(\cdot,\eta)\|_\infty^2 + e^{-C}\sqrt{k} \)}\ge \frac{\frac{1}{4}\|\psi\|_1^2}{\frac{1}{k}+\frac{2}{n} + e^{-C}\frac{\sqrt{k}}{n} }.
	\end{equation}
	In order to conclude, we need to show that the  eigenfunction $\psi$ of $-L^{\bin(1)}$ satisfies
	\begin{equation}\label{eq:assumption_eigenfunctions}
	\liminf_{n\to\infty}\left\| \psi\right\|_1 > 0\ ,
	\end{equation}
	which, by H\"older inequality and $\left\| \psi\right\|_2=1$, holds if we prove that
	\begin{equation}\label{eq:limsup}
	\limsup_{n\to\infty}\left\| \psi\right\|_\infty <\infty\ .
	\end{equation}
	The above follows by noting that, for all $t\ge 0$ and $x \in V$,
	\begin{align}\label{eq:trick_sup}
	|\psi(x)|=&\ e^{\frac{t}{t_\rel}}\left| S_t^{\bin(1)}\psi(x)\right|
	=\ e^{\frac{t}{t_\rel}}|\pl h_t^x,\psi \pr |
	\le \ e^{\frac{t}{t_\rel}}\|h_t^x\|_2 \|\psi\|_2,
	\end{align}
	where we  used Cauchy-Schwarz inequality. 
	Hence, plugging $t=Ct_\rel$ into \cref{eq:trick_sup}, using \cref{assumption:nash} and recalling that $\|\psi \|_2=1$, \cref{eq:limsup} holds and the desired claim follows. 			
\end{proof}
Similarly to what has been done at the end of \cref{sec:avg-results-proofs}, we show how to use our arguments to get some quantitative bounds on the \ac{TV} mixing of the particle system in the case in which $k=\omega(n^2)$.
\begin{proposition}[Pre-cutoff for the Binomial Splitting]\label{th:precutoff-bin}
	Consider a sequence of graphs and site-weights such that \cref{assumption:nash,assumption:unif_ellipticity} hold. Then, if $k/n^2\to\infty$, for all  $\delta\in(0,1)$,
	\begin{equation}
	\limsup_{n\to\infty} \dtv(T^+)\le \delta,\qquad
	\liminf_{n\to\infty} \dtv(T^-)\ge 1-\delta,
	\end{equation}
	where 
	\begin{equation}
	T^+\coloneqq a\, \frac{t_\rel}{2}\log(k) + C t_\rel , \qquad T^-\coloneqq b\,\frac{t_\rel}{2}\log(k) - C t_\rel,
	\end{equation}
	\begin{equation}
	a\coloneqq 2\,\frac{\log(k/n)}{\log(k)}\in[1,2],\qquad b\coloneqq 2\,\frac{\log(n)}{\log(k)}\in(0,1],
	\end{equation}
	and  some $C=C(c_{\rm ell},c_{\rm ratio},\delta)>0$. 
\end{proposition}
Let us observe that, compared to \cref{th:precutoff-avg}, the effect of $k=\omega(n^2)$ alters not only the first order of $T^+$, but also that of $T^-$.
\begin{proof}
	The lower bound follows directly by \cref{pr:wilson-final}, while, for the upper bound, it is enough to combine \cref{lemma:ub-general-initial} and the estimate in \cref{eq:last-est}.
\end{proof} 
\bigskip

%%%%%%%%%%%%%%%%%%%%%%%%%%%%%%%%%%%%%%%%%%%%%%
%% Support information, if any,             %%
%% should be provided in the                %%
%% Acknowledgements section.                %%
%%%%%%%%%%%%%%%%%%%%%%%%%%%%%%%%%%%%%%%%%%%%%%
\begin{acks}[Acknowledgments]
The authors wish to thank Pietro Caputo and Jan Maas for several fruitful discussions.
\end{acks}

%%%%%%%%%%%%%%%%%%%%%%%%%%%%%%%%%%%%%%%%%%%%%%
%% Funding information, if any,             %%
%% should be provided in the                %%
%% funding section.                         %%
%%%%%%%%%%%%%%%%%%%%%%%%%%%%%%%%%%%%%%%%%%%%%%
\begin{funding}
	M.Q. was supported by the European Union’s Horizon 2020 research and innovation programme
	under the Marie Skłodowska-Curie grant agreement no.\ 945045, and by the NWO Gravitation
	project NETWORKS under grant no.\ 024.002.003.
Part of this work was completed while being a member of GNAMPA-INdAM and of COST Action GAMENET, and receiving partial support by the GNAMPA-INdAM Project 2020 ``Random walks on random games'' and PRIN 2017 project ALGADIMAR. 

F.S.\  gratefully acknowledges funding by the Lise Meitner fellowship, Austrian Science Fund (FWF):
M3211. Part of this work was completed while funded by the European Union’s Horizon 2020 research and innovation
programme under the Marie-Sk\l{}odowska-Curie grant agreement No.\ 754411.
\end{funding}

\end{document}